\documentclass[12pt,reqno]{amsart}

\usepackage{graphicx}%
\usepackage{multirow}%
\usepackage{amsmath,amssymb,amsfonts}%
\usepackage{amsthm}%
\usepackage{mathrsfs}%
\usepackage{hyperref}
\usepackage[english]{babel}
\usepackage{enumitem}
\usepackage{float}
\usepackage{lipsum}
\usepackage[foot]{amsaddr}
\usepackage{csquotes}

\usepackage{tikz}

\usepackage{comment}
\usepackage{bbm}

\usepackage[doi=false,isbn=false,url=false,eprint=false, giveninits=true, maxbibnames=99]{biblatex}

\newcommand{\pw}{\mathbb{P}(\omega)}

\newtheorem{theorem}{Theorem}
\numberwithin{theorem}{section}

\newtheorem{proposition}[theorem]{Proposition}% 
\newtheorem{definition}[theorem]{Definition}%
\newtheorem{remark}[theorem]{Remark}%

\newtheorem{lemma}[theorem]{Lemma}
\newtheorem{corollary}[theorem]{Corollary}

\usepackage[dvipsnames]{xcolor}

\usepackage{enumitem}
\setenumerate[1]{label={(\arabic*)}}

\DeclareUnicodeCharacter{0306}{{\u i}}

\title[]{The transition between synchronization and chaos for random Blaschke products}

\author{Cecilia González-Tokman}
\email{cecilia.gt@uq.edu.au}
\address{School of Mathematics and Physics, The University of Queensland, Brisbane, QLD 4072, Australia}
\author{Renee Oldfield}
\email{renee.oldfield@uq.edu.au}

\begin{document}
\begin{abstract}
In this paper, we establish sufficient conditions for which a class of random circle maps, the so-called random Blaschke products, exhibits either synchronization or chaotic dynamics. We also establish necessary and sufficient conditions for such systems to have a unique random fixed point attractor in the unit disc. As an example, for a class of two map cocycles,
we identify the random physical measures as a parameter is varied and investigate the critical value where the dynamics transitions from order to chaos.
\end{abstract}

\maketitle

\section{Introduction}

Understanding routes to chaos is a fundamental question in dynamical systems. In the deterministic setting, chaotic dynamics has been extensively studied, and different routes between order and chaos in discrete time systems have been identified. Ruelle and Takens show that, as a parameter is varied, a physical system can transition from a stable fixed point to periodic motion, then to quasiperiodic motion, and finally to the emergence of a strange attractor \cite{ruelle_takens}. Another route is period-doubling, where varying a parameter sees a stable period-2 orbit loses stability and orbits of increasing period ($2,4,8,...$) appear, and further variation sees chaotic dynamics emerge \cite{feigenbaum}. Another is intermittency, introduced by Pomeau and Manneville, where long periods of quasiperiodic behavior are followed by irregular bursts increasing in frequency and duration as a parameter is varied \cite{intermittency}. However, this transition is much less understood in the random setting, where the evolution of the system is more complex, as it incorporates aspects and seasonality of the environment. In this paper, we describe the transition between order and chaos for random Blaschke products and give sufficient conditions for chaotic dynamics and synchronization on the unit circle.

Dynamics generated by holomorphic self-maps of the disc and other simply connected, open subsets of the complex plane have been extensively studied due to their deep connections with complex analysis and geometric function theory. Blaschke products are a family of holomorphic functions of the unit disk, preserving the unit circle, which provide a rich class of examples of discrete time dynamical systems. Some of the most popular and well-studied one-dimensional maps belong to this class, including rotations, the doubling map and some M\"obius transformations. Furthermore, this family contains a variety of maps with rich nonlinear dynamics, ranging from expanding maps with absolutely continuous invariant measures, to maps with a sink. Hence, Blaschke products provide a natural setting to investigate transitions between order and chaos.

In this paper, we study the asymptotic behavior and random physical invariant measures for a family of admissible Blaschke product cocycles (see Definition \ref{admissible}) on $\mathbb{T}.$ We denote by $(\mathcal{T}, \sigma,\mathbb{P})$ the cocycle on $\hat{\mathbb{C}}$ generated by $\mathcal{T}=(T_{\omega})_{\omega\in\Omega}$ with $\mathbb{P}$-preserving driving $\sigma.$ Sometimes, we refer to this cocycle simply as $\mathcal{T},$ when $\sigma$ and $\mathbb{P}$ are clear from the context. Let $\mathcal{T}^{*}:=\mathcal{T}|_{\mathbb{T}}.$ We show that the cocycle $\mathcal{T}^{*}$ exhibits chaotic behavior when $\mathcal{T}$ is \textit{eventually expanding on average}, and give sufficient conditions for synchronization of trajectories in the absence of such on-average expansion. Our main theorem is as follows. \begin{theorem}\label{possible_limits}
Let $(\mathcal{T},\sigma,\mathbb{P})$ be an admissible Blaschke product cocycle. Then, exactly one of the following two cases hold:
\begin{enumerate}
    \item $\mathcal{T}$ is eventually expanding on average. In this case, $\mathcal{T}^{*}$ has a unique absolutely continuous random invariant measure $\mu:=\{\mu_{\omega}\}_{\omega\in\Omega}$ with marginal $\mathbb{P},$ such that for $\mathbb{P}$-a.e. $\omega\in\Omega,$ $\frac{d\mu_{\omega}}{d\text{Leb}}=P_{x_{\omega}},$ where $x_{\omega}\in D$ is the (unique) accumulation point of $T_{\sigma^{-n}\omega}^{(n)}(z)$,  $\forall z\in D,$ and $P_{x_{\omega}}$ is the Poisson kernel associated to $x_{\omega}.$  Furthermore, $\mu$ is the unique random physical measure for $\mathcal{T}^{*}$ and its basin has full Lebesgue measure. 
    \item For $\mathbb{P}$-a.e. $\omega\in\Omega,$ $T_{\sigma^{-n}\omega}^{(n)}(z), z\in D$ has an accumulation point on $\mathbb{T}.$ Moreover, if $\int_{\Omega}\log\left(1-|T_{\omega}(0)|\right)\,d\mathbb{P}(\omega)>-\infty$ and $$\lim_{n\to\infty}\frac{1}{n}\log\left(1-|T_{\sigma^{-n}\omega}^{(n)}(0)|\right)\neq0$$ for a positive measure subset of $\Omega,$ then there exists measurable $x:\Omega\to\mathbb{T}$ such that, for $\mathbb{P}$-a.e. $\omega,$ all $z\in D$ and $\text{Leb}$-a.e. $z\in\mathbb{T},$ taking $x_{\omega}:=x(\omega),$  $\lim_{n\to\infty}T_{\sigma^{-n}\omega}^{(n)}(z)=x_{\omega}.$ The measure $\mu:=\{\mu_\omega\}_{\omega\in\Omega}$ with marginal $\mathbb{P}$ such that $\mu_{\omega}=\delta_{x_{\omega}}$ is the unique random physical measure for $\mathcal{T}^{*}$ and its basin has full Lebesgue measure. \label{case2}

\end{enumerate}
\end{theorem}

We illustrate our results for a class of two map cocycles (Theorem \ref{theorem_example}). We describe the asymptotic behavior of orbits starting in the disc and on the circle, and their random invariant (physical) measure, as the probability of choosing each map varies. In addition, we show that at the critical parameter value where the dynamics on the circle transitions from chaos to synchronization, iid random compositions exhibit a phenomenon not present in the deterministic case. For a single map, trajectories inside the disc entering an arbitrarily small neighborhood of the boundary fixed point necessarily forces convergence to that point. In contrast, in an iid setting, we prove that for almost every sequence of maps and every sufficiently small neighborhood (in $D$) of the common fixed point on the circle, forward trajectories in the unit disc escape and return to an arbitrarily small neighborhood of the fixed point intermittently.  We also show there is a transition between positive and zero measure theoretic entropy for this family at this critical probability value.

 Understanding how processes transition from chaos to synchronization is a key problem in the study of random dynamical systems. In \cite{synchronization_intermittency}, Abbasi, Gharaei and Homburg give sufficient conditions for synchronization in iterated function systems generated by finitely many logistic maps and establish intermittent dynamics in the presence of two particular logistic maps. Homburg and Kalle analyze the two-point motion of iterated function systems
on the unit interval generated by expanding and contracting
aﬃne maps, and show intermittent dynamics when the Lyapunov exponent is zero, chaotic dynamics when it is positive, and synchronization when it is negative \cite{homburg_kalle_intermittency}. For a family of circle endomorphisms with additive noise, Goverse, Homburg, and Lamb classify the two-point dynamics and the random invariant measure, depending on the sign of the Lyapunov exponent, and establish intermittency in the two-point dynamics for the zero Lyapunov exponent case \cite{random_circle_endo}. Yan, Majumdar, Ruffo, Sato, Beck, and Klages show the transition between chaotic dynamics and regular dynamics for the iid composition of a contracting map and expanding map on the interval \cite{simple_random_map}.

Random dynamical systems on one‑dimensional spaces, such as those generated by Möbius transformations, matrix cocycles, and circle homeomorphisms, have deep connections to many areas, including continued fractions and random matrix theory. Furstenberg establishes the existence of a unique stationary measure, referred to as the Furstenberg measure, for the iid product of $SL_2(\mathbb{R})$ matrices with no finite invariant set \cite{furstenberg}. 

Brin and Kifer consider the Markov chain on a compact manifold generated by a sequence of random diffeomorphisms and construct non-random stable foliations for these systems \cite{kifer_random_diffeo}. In \cite{homburg_random_diffeo}, Homburg and Zmarrou study random circle diffeomorphisms with identically distributed noise. They provide sufficient conditions for the existence of an attracting random fixed point or attracting random periodic orbit and investigate the bifurcation between them. Homburg gives conditions for synchronization for random diffeomorphisms on arbitrary compact manifolds, extending results for diffeomorphisms on the circle \cite{homburg_synchronization}. In \cite{malicet}, Malicet shows that iid composition of homeomorphisms of the circle almost surely contract small intervals, under the assumption that there is no probability measure left invariant by almost every map. Gelfert and Salcedo prove several statistical properties for random dynamical systems on a compact metric space, including a central limit theorem, the strong law of large numbers,
and the law of the iterated logarithm. In addition, they establish synchronization rates and large deviations for an iterated function system of circle diffeomorphisms \cite{gelfert_synchronization}. 

One particular type of random diffeomorphism, random Möbius transformations, have been extensively studied. In \cite{Mobius_plane}, Ambroladze and Wallin provide sufficient conditions for the almost sure convergence to $\overline{\mathbb{R}}$ of iid sequences of Möbius transformations mapping the upper half plane to itself. Karmakar and Key give sufficient conditions for sequences of several different families of Möbius transformations to converge to a point or to converge in distribution \cite{iid_Mobius}. In \cite{Mobius_composition}, Jacques and Short provide necessary and sufficient conditions for convergence of sequences of Möbius transformations mapping the unit disc strictly inside itself and calculate the Hausdorﬀ dimension of the set of sequences of so-called
limit-disc type in some cases. In \cite{hochman_furstenberg}, Hochman and Solomyak give results on the dimension
of the Furstenberg measure for iid products of $SL_{2}(\mathbb{R})$ matrices.

Blaschke products, comprised of products of Möbius transformations mapping the unit disc to itself, are a natural generalization. The dynamics of Blaschke products have been extensively studied in the deterministic setting. For Blaschke products of degree at least two, the orbits for all points in the unit disc $D,$ $\hat{\mathbb{C}}\backslash\overline{D}$ tend to the indifferent or attracting fixed point in $\overline{D}, \hat{\mathbb{C}}\backslash D,$ respectively. The more interesting dynamics occur on the unit circle, where the dynamics fall into one of four cases: elliptic, doubly parabolic, singly parabolic, and hyperbolic. In each of these cases, the Julia and Fatou sets of the map  can be explicitly described \cite{baker_domains2, baker_domains,  parabolic, fletcher}, and an invariant measure can be constructed \cite{ergodic_inner_function, martin, expandPujals}, depending on the attracting or indifferent fixed point in $\overline{D}$. Doering and Mañé show that for an inner function $T$ the Dirac measure supported on an attracting or indifferent fixed point on $\mathbb{T}$ is the unique physical measure, even when $T$ is recurrent \cite{Mane1991}.

Much work has gone into a similar classification for non-autonomous iteration of maps of the unit disc or other domains. We refer to \cite{beardon, keen_lakic, lorentzen, lorentzen2}, and the references therein, for an overview of classical results in the theory of iterated function systems consisting of contractions of a simply connected (open) domain of $\mathbb{C}.$ In \cite{benini_1}, Benini, Evdoridou, Fagella, Rippon and Stallard give sufficient conditions for convergence to the boundary of sequences of non-autonomous sequences of holomorphic maps between simply connected domains and sufficient conditions for the so-called Denjoy-Wolff set, a generalization of the Denjoy-Wolff point, to have either full or zero measure. In \cite{benini2}, the same authors generalize the Poincaré recurrence theorem and sharpen some of their results from \cite{benini_1} for the forward composition of inner functions.

Another topic of interest is whether a random dynamical system has a (unique) physical invariant measure(s), which describes the long-term distribution of the system,  and what form these take. Buzzi establishes the existence of absolutely continuous SRB measures for random expanding on average Lasota-Yorke maps under certain conditions in \cite{buzzi_SRB} and the uniqueness of such a measure under a covering condition in \cite{buzzi_edoc}. Atnip, Froyland, González-Tokman and Vaienti develop a quenched thermodynamic formalism for random dynamical systems that satisfy a random covering condition in \cite{AFGV2}. For a family of diffeomorphisms with random noise, Araújo shows the existence of the time average for almost every orbit and that these averages coincide with a finite number of absolutely continuous physical measures \cite{araujo_time_average}. Barrientos, Nakamura, Nakano and Toyokawa give conditions for the existence and finitude of absolutely continuous ergodic stationary probability measures for the iid composition of measurable maps on a Polish space \cite{finititude_physical_measure}. González-Tokman and Quas establish the existence of unique absolutely continuous invariant measures for a class of uniformly expanding random Blaschke products and describe the Lyapunov spectrum of the associated cocycle of Perron-Frobenius
operators acting on Banach spaces of analytic functions on an annulus \cite{CGTAQ}. In \cite{preprint}, we describe the absolutely continuous random invariant measure and associated measure theoretic entropy for a class of admissible expanding on average random Blaschke products.

 In this work, we show that an admissible Blaschke product cocycle has a so-called random fixed point in the disc, and absolutely continuous random invariant (physical) measure, if it is eventually expanding on average. We also give sufficient conditions in the not eventually expanding on average setting, based on the work of Gouzël and Karlsson \cite{subadditive_MET} and Benini, Evdoridou, Fagella, Rippon and Stallard \cite{benini_1}, for the cocycle to have a random fixed point on the unit circle and a physical measure which is Dirac on the fibers. 

 In Section \ref{main_section}, we prove Theorem \ref{possible_limits}. We also prove several results relating to boundary convergence, that is, \ref{possible_limits}$(2),$ and show that eventually expanding on average $\mathcal{T}$ has positive measure theoretic entropy. In Section \ref{example_section}, we investigate the disc and boundary dynamics of a 2-map cocycle as a parameter is varied. We explicitly describe the asymptotic dynamics inside the disc and on the circle for non-critical parameter values, and investigate the dynamics at the critical value where the dynamics transitions from chaos to order.

\section*{Acknowledgements}

The authors thank Dmitry Dolgopyat, Enrique Pujals and Anthony Quas for insightful discussions. They also thank MATRIX for hosting them during the workshop “Statistical Properties and Extremes in Dynamical Systems Theory and Numerics" in January 2026, where part of this work was completed. CGT and RO acknowledge support from the Australian Research Council. RO acknowledges the Australian Government Research Training Program
for financial support. CGT acknowledges Lorenzo Díaz and PUC Rio for their hospitality during a visit in 2023, where this project started.

\section{Asymptotic behavior for admissible Blaschke product cocycles}\label{main_section}

\subsection{Notation and preliminaries}

In this section, we give relevant definitions and results that will be used through this work.

\begin{definition}[Blaschke product cocycle] 
Let $\sigma$ be an invertible, ergodic, measure-preserving transformation of a probability space $(\Omega, \mathcal{F}, \mathbb{P}).$ Let $\mathcal{T}=(T_{\omega})_{\omega\in\Omega},$ where
     for each $\omega\in\Omega,$ we let   $T_{\omega}:\hat{\mathbb{C}}\to\hat{\mathbb{C}}$ be given by $$T_{\omega}(z)=\rho_{\omega}\prod_{i=1}^{d_{\omega}}\frac{z-a_{i,\omega}}{1-\bar{a}_{i,\omega}z},$$ where  $d:\Omega\to\mathbb{N},\rho:\Omega\to\mathbb{T}$ are measurable and for each $j\ge1$ $\Omega_{j}:=\{\omega\in\Omega:d_{\omega}=j\}$ is measurable. Furthermore, when $\Omega_j$ is non-empty, $a:\Omega_{j}\to D^{j}, D^{j}:=\underbrace{D\times\cdots\times D}_{j \text{ times}}$ is measurable. The Blaschke product cocycle generated by $(\mathcal{T}, \sigma, \mathbb{P}),$ or simply $\mathcal{T},$ is defined by $T_{\omega}^{(0)}=\text{Id}$ and $T_{\omega}^{(n)}=T_{\sigma^{n-1}\omega}\circ\cdots\circ T_{\omega}$ for all $n\in\mathbb{N}.$
\end{definition}

Throughout this work, we consider a class of Blaschke product cocycles that we call admissible. $\text{Leb}$ will denote $\text{Leb}|_{\mathbb{T}},$ the Lebesgue measure restricted to the unit circle.
\begin{definition}\label{admissible}
    We will call a Blaschke product cocycle $\mathcal{T}$ \textit{admissible} if \begin{enumerate}
    \item $\mathbb{P}(\omega\in\Omega:\deg(T_{\omega})\ge2)>0;$
    \item $\int_{\Omega}\log\frac{\deg(T_{\omega})}{\inf_{z\in\mathbb{T}}|T_{\omega}'(z)|}\,d\mathbb{P}(\omega)<\infty;$
    \item $\int_{\Omega}\log^{+}\left(\int_{\mathbb{T}}\left|\frac{T_{\omega}''}{(T_{\omega}')^2}\right|\,d\text{Leb}\right)\,d\mathbb{P}(\omega)<\infty;$ and
    \item $\int_{\Omega}\log\inf_{z\in\mathbb{T}}|T_{\omega}'(z)|\,d\mathbb{P}(\omega)>-\infty.$
    \end{enumerate} If, in addition, $$\int_{\Omega}\log\inf_{z\in\mathbb{T}}|T_{\omega}'(z)|\,d\mathbb{P}(\omega)>0,$$ then we call $\mathcal{T}$ admissible expanding on average (EOA). If $$\lim_{n\to\infty}\int_{\Omega}\frac{1}{n}\log\inf_{z\in\mathbb{T}}|(T_{\omega}^{(n)})'(z)|\,d\mathbb{P}(\omega)>0,$$ then we call $\mathcal{T}$ eventually EOA.
\end{definition}

We will refer to a $\mathcal{T}$-invariant measure, or random invariant measure,  given by the following definition.
\begin{definition}[Random invariant measure]
A map $\mu:\mathcal{F}\times\Omega\to[0,1]$ given by $\mu(F,\omega)=\mu_{\omega}(F)$ satisfying \begin{itemize}
    \item for every $F\in\mathcal{F}, \omega\mapsto\mu_{\omega}(F)$ is measurable; and
    \item for $\mathbb{P}$-a.e. $\omega\in\Omega, F\mapsto\mu_{\omega}(F)$ is a Borel probability measure
\end{itemize} is a random (probability) measure. We say $\mu$ has disintegration $\{\mu_{\omega}\}_{\omega\in\Omega}$ with marginal $\mathbb{P}.$

 Let $(\mathcal{T}, \sigma, \mathbb{P})$ be a map cocycle. $\mu$ is a random invariant measure, or $\mathcal{T}$-invariant measure, if  $\mu_{\sigma\omega}=(T_{\omega})_{*}(\mu_{\omega})$ for $\mathbb{P}$-a.e. $\omega\in\Omega.$
\end{definition}
Throughout, we will often take either $X=\mathbb{T}$ or $X=D.$

We will use the following theorem in the proof of Theorem \ref{possible_limits}$(1)$ and Lemma \ref{N_admissible_rfp}.

\begin{theorem}[Theorem 1 \cite{preprint}]\label{buzzi_blaschke}

    Let $(\mathcal{T}, \sigma, \mathbb{P})$ be an admissible expanding on average Blaschke product cocycle.
    Then, there exists a unique random absolutely continuous invariant measure (acim) $\mu,$ with disintegration $\{\mu_{\omega}\}_{\omega\in\Omega}$ with marginal $\mathbb{P},$ such that \begin{enumerate}[label=(\roman*)]
        \item There exists a measurable map $x:\Omega\to D$ such that for $\mathbb{P}$-a.e. $\omega\in\Omega,$ $x_{\omega}=\lim_{n\to\infty}T_{\sigma^{-n}\omega}^{(n)}(z)$ for all $z\in D,$ and $\frac{d\mu_{\omega}}{d\text{Leb}}=P_{x_{\omega}},$ where $P_{x_{\omega}}  (z)=\frac{1-|x_{\omega}|^2}{|z-x_{\omega}|^2}, z\in\mathbb{T}.$ Then, $(T_{\omega})_{*}(\mu_{\omega})=\mu_{\sigma\omega}$ for $\mathbb{P}$-a.e. $\omega\in\Omega;$ and \label{main_thm_rfp} 
        \item For $\mathbb{P}$-a.e. $\omega\in\Omega$ and any $h_{*}\in\text{BV}(\mathbb{T})$ where $||h_{*}||_{1}=1,$ there exists $n_{0}(\omega)\in\mathbb{N}$ and $\rho<1$ such that for all $n\ge n_{0}(\omega),$ $\sup_{||f||_1=1}\left|\int_{\mathbb{T}}f\circ T_{\sigma^{-n}\omega}^{(n)}\cdot h_{*}\,d\text{Leb}-\int_{\mathbb{T}}f P_{x_{\omega}}\,d\text{Leb}\right|\le\rho^{n}.$ \label{main_thm_diff_bound}

    \end{enumerate}
\end{theorem}

 The following lemma will enable us to extend Theorem \ref{buzzi_blaschke} to an admissible eventually EOA cocycle. \begin{lemma}\label{N_EOA_admissible}
    If $\mathcal{T}$ is admissible and eventually EOA, that is, $$\lim_{n\to\infty}\int_{\Omega}\frac{1}{n}\log\inf_{z\in\mathbb{T}}|(T_{\omega}^{(n)})'(z)|d\mathbb{P}(\omega)>0,$$ then $\mathcal{T}^{(N)}$ is an admissible EOA cocycle for all $N\ge N_{0}$ for some $N_{0}\in\mathbb{N}.$
\end{lemma}
\begin{proof}
     Since $\lim_{n\to\infty}\int_{\Omega}\frac{1}{n}\log\inf_{z\in\mathbb{T}}|(T_{\omega}^{(n)})'(z)|d\mathbb{P}(\omega)>0$, there exists $N_{0}\in\mathbb{N}$ such that $f_{N_0}:=\int_{\Omega}\frac{1}{N_0}\log\inf_{z\in\mathbb{T}}|(T_{\omega}^{(N_0)})'(z)|d\mathbb{P}(\omega)>0.$ Since $f_n$ is subadditive, it follows that $f_N>0$ for all $N\ge N_{0}.$  In particular, $$Nf_N=\int_{\Omega}\log\inf_{z\in\mathbb{T}}|(T_{\omega}^{(N)})'(z)|d\mathbb{P}(\omega)>0.$$ One can show that if $\mathcal{T}$ is admissible, then for every fixed $N\in\mathbb{N},$ $\mathcal{T}^{(N)}$ is admissible (see Appendix \ref{N_admissible}), and thus the claim follows.
\end{proof}

\subsection{Pullback limits}\label{section_rfp}

In this section, we present several results concerning the limit points of the pullback $T_{\sigma^{-n}\omega}^{(n)}(z)$ for $z\in D,$ which will be used in Section \ref{section_EEOA} and \ref{section_proof_main_thm}.

\begin{definition}[Random attracting fixed point]
    Suppose there is measurable $c:\Omega\to\overline{D}$ such that for $\mathbb{P}$-a.e. $\omega\in\Omega$ and all $z\in D,$ $$\lim_{n\to\infty}T_{\sigma^{-n}\omega}^{(n)}(z)=c_{\omega}.$$ We call $\{c_{\omega}\}_{\omega\in\Omega}$ a random attracting fixed point of $\mathcal{T}$. 
\end{definition}

\begin{lemma}\label{N_admissible_rfp}
    Suppose admissible $\mathcal{T}$ is eventually EOA. Then, for $\mathbb{P}$-a.e. $\omega\in\Omega,$ $\mathcal{T}$ has a random attracting fixed point $\{c_{\omega}\}_{\omega\in\Omega}$ in $D.$   Let $\mu$ be the random measure that has disintegration $\{\mu_{\omega}\}_{\omega\in\Omega}$ with marginal $\mathbb{P},$ where $\frac{d\mu_{\omega}}{d\text{Leb}}=P_{c_{\omega}}.$ Then, $\mu$ is a random invariant measure for $\mathcal{T}^{*}.$
\end{lemma}
\begin{proof}
From Lemma \ref{N_EOA_admissible}, we know that for some $N\in\mathbb{N},$ $\mathcal{T}^{(N)}$ is admissible EOA. Let $\hat{\sigma}=\sigma^{N}.$ It follows from Theorem \ref{buzzi_blaschke} applied to $\mathcal{T}^{(N)}$ that, for $\mathbb{P}$-a.e. $\omega\in\Omega$ and all $z\in D,$ $\lim_{n\to\infty}T_{\hat{\sigma}^{-n}\omega}^{(n)}(z)=\hat{c}_{\omega}\in D.$ For $k\in\mathbb{Z},$ take $m\in\mathbb{Z}$ and $0\le r\le N-1$ so that $k=mN+r.$ It follows that, for $\mathbb{P}$-a.e. $\omega\in\Omega$ and all $z\in D,$ $\lim_{n\to\infty}T_{\sigma^{-n}\omega}^{(n)}(z)=c_{\omega},$  where $c_{\sigma^{k}\omega}=T_{\sigma^{mN}\omega}^{(r)}(\hat{c}_{\hat{\sigma}^{r}\omega}).$
Note that $c_{\sigma\omega}=T_{\omega}(c_{\omega}).$ It follows that the measure $\mu$ with disintegration $\{\mu_{\omega}\}_{\omega\in\Omega}$ with marginal $\mathbb{P},$ where $\frac{d\mu_{\omega}}{d\text{Leb}}=P_{x_{\omega}},$ is a random invariant measure for $\mathcal{T}^{*}.$
\end{proof}

\begin{definition}
    Given an admissible Blaschke product cocycle $(\mathcal{T}, \sigma, \mathbb{P})$ where  $\mathcal{T}:=(T_{\omega})_{\omega\in\Omega},$ we define $\Omega_{D}$ and $\Omega_{\mathbb{T}}$ as follows.
    \begin{align*}
        &\Omega_{D}:=\{\omega\in\Omega:(T_{\sigma^{-n}\omega}^{(n)}(0))_{n=1}^{\infty}\text{ has a limit point in } D\};\text{ and }\\&\Omega_{\mathbb{T}}:=\{\omega\in\Omega:(T_{\sigma^{-n}\omega}^{(n)}(0))_{n=1}^{\infty}\text{ has a limit point in  } \mathbb{T}\}.  
    \end{align*}
\end{definition}

\begin{proposition}\label{constant_full_measure}
    For a Blaschke product cocycle $\mathcal{T},$ $\Omega_{D}$ and $\Omega_{\mathbb{T}}$ are measurable, $\Omega=\Omega_{D}\cup\Omega_{\mathbb{T}},$ and each of the sets has either full or zero $\mathbb{P}$-measure. 
\end{proposition}
\begin{proof}
    By assumption,  $\omega\mapsto T_{\omega}(0)=\rho_{\omega}\prod_{i=1}^{d_{\omega}}a_{i,\omega}$ is measurable. First, we show $\Omega_{D}$ is measurable. Let $z_{\omega,n}:=T_{\sigma^{-n}\omega}^{(n)}(0)$ and $D_{R}$ denote the disc of radius $R$ centered at the origin. For fixed $n\in\mathbb{N}$ and $R<1,$ it follows that $$\{\omega\in\Omega:z_{\omega, n}\in D_{R}\}$$ is measurable. Thus,  $$\bigcap_{n=1}^{\infty}\bigcup_{k=n}^{\infty}\{\omega\in\Omega:z_{\omega, k}\in D_{R}\}$$ is measurable. Let $R_{j}=1-\frac{1}{j}.$ We have that $$\bigcup_{j=1}^{\infty}\bigcap_{n=1}^{\infty}\bigcup_{k=n}^{\infty}\{\omega\in\Omega:z_{\omega, k}\in D_{R_j}\}=\Omega_{D}$$ is measurable.
    Similarly, it follows that $$\Omega_{\mathbb{T}}=\bigcap_{j=1}^{\infty}\bigcap_{n=1}^{\infty}\bigcup_{k=n}^{\infty}\{\omega\in\Omega:z_{\omega,k}\in D\backslash D_{R_{j}}\}$$ is measurable.
    
    For all $\omega\in\Omega$ and $n\in\mathbb{N},$ $|T_{\sigma^{-n}\omega}^{(n)}(0)|<1.$ It follows that $T_{\sigma^{-n}\omega}^{(n)}(0)$ has a subsequence converging to a point in $\overline{D}.$ Thus, $\Omega=\Omega_{D}\cup\Omega_{\mathbb{T}}$ and $\mathbb{P}(\Omega_{D}\cup\Omega_{\mathbb{T}})=1.$

    Let us first suppose $\omega\in\Omega_{D},$ so there exists a subsequence $n_k:=n_k(\omega)$ such that $T_{\sigma^{-n_k}\omega}^{(n_k)}(0)$ converges some $c_{\omega}\in D.$ Now consider $\sigma\omega.$  We have $$T_{\sigma^{-(n_k+1)}(\sigma\omega)}^{(n_k+1)}(0)=T_{\omega}\circ T_{\sigma^{-n_k}\omega}^{(n_k)}(0).$$ Since $T_{\sigma^{n_k}\omega}^{(n_k)}(0)$ converges to $c_{\omega}\in D$ as $k\to\infty$ and $T_{\omega}$ is a finite Blaschke product, it follows that $T_{\sigma^{-(n_k+1)}(\sigma\omega)}^{(n_k+1)}(0)$ converges to $d_{\omega}=T_{\omega}(c_{\omega})\in D.$ Thus, $T_{\sigma^{-n}(\sigma\omega)}^{(n)}(0)$ has a subsequence converging to $d_{\omega}\in D$ and $\sigma\omega\in\Omega_{D}.$ Since $\sigma$ is ergodic, invertible and $\mathbb{P}$-preserving and $\Omega_{D}$ is measurable, either $\mathbb{P}(\Omega_{D})=0$ or  $\mathbb{P}(\Omega_{D})=1.$ Replacing $c_{\omega}\in D$ with $c_{\omega}\in\mathbb{T},$ the proof follows similarly for $\Omega_{\mathbb{T}}.$ 
\end{proof}

Recall $d_{D},$ the hyperbolic distance on $D,$ where $d_{D}(y,z)=\log\left(\frac{1+\left|\frac{z-y}{1-\bar{y}z}\right|}{1-\left|\frac{z-y}{1-\bar{y}z}\right|}\right)$ for $y,z\in D.$
\begin{theorem}[Schwarz-Pick Lemma \cite{hyperbolic_metric}]\label{scwarz_pick}
    Suppose that $T:D\to D$ is holomorphic.  Then either
    \begin{enumerate}
        \item $T$ is a hyperbolic contraction, that is, for all $y,z\in D,$ $$d_{D}(T(y),T(z))<d_{D}(y,z);$$ or
        \item $T$ is a conformal automorphism of $D$ and, for all $y,z\in D,$ $$d_{D}(T(y),T(z))=d_{D}(y,z).$$
    \end{enumerate}
\end{theorem}
\begin{remark}
    The conformal automorphisms of the unit disc are maps of the form $T(z)=\theta\frac{z-a}{1-\bar{a}z},$ where $\theta\in\mathbb{T}, a\in D,$ that is, Blaschke products of degree 1. Any Blaschke product of degree at least 2 is a hyperbolic contraction.
\end{remark}

In the next proposition, we show that if $\mathbb{P}(\Omega_{\mathbb{T}})=0,$ then, for $\mathbb{P}$-a.e. $\omega\in\Omega,$ $T_{\sigma^{-n}\omega}^{(n)}(0)$ converges to a (unique) measurable limit in $D.$ This proposition will be used in the proof of Proposition \ref{rfp_oi_eoa}.
\begin{proposition}\label{no_circle_pt}
    Suppose $\mathcal{T}$ is admissible and $\mathbb{P}(\Omega_{\mathbb{T}})=0.$ Then, there exists measurable $c:\Omega\to\overline{D}$ such that for $\mathbb{P}$-a.e. $\omega\in\Omega,$ $\lim_{n\to\infty}T^{(n)}_{\sigma^{-n}\omega}(0)=c_{\omega}\in D$ and $c_{\sigma\omega}=T_{\omega}(c_{\omega}).$ 
\end{proposition}
\begin{proof}
    Take $\omega\in\Omega\backslash\Omega_{\mathbb{T}}.$ Let $L(\omega)$ denote the set of limit points of $T_{\sigma^{-n}\omega}^{(n)}(0)$ in $D.$ Since $\omega\not\in\Omega_{\mathbb{T}},$ it follows that $L(\omega)$ is a non-empty closed subset of $D$, and there exists $r_{\omega}<1$ such that for all $n\in\mathbb{N},$ $T_{\sigma^{-n}\omega}^{(n)}(0)\in\overline{D}_{r_{\omega}},$ and $L(\omega)\subseteq\overline{D}_{r_{\omega}}.$ Now consider $\sigma\omega.$ We have $T_{\sigma^{-n}(\sigma\omega)}^{(n)}(0)=T_{\omega}(T_{\sigma^{-(n-1)}\omega}^{(n-1)}(0))$ and so, if $x\in L(\omega)$ then $T_{\omega}(x)\in L(\sigma\omega).$  Thus, $T_{\omega}(L(\omega))\subseteq L(\sigma\omega).$ Now consider $y\in L(\sigma\omega).$ There exists $n_k$ such that $$y=\lim_{k\to\infty}T_{\sigma^{-n_k}(\sigma\omega)}^{(n_k)}(0)=\lim_{k\to\infty}T_{\omega}(T_{\sigma^{-(n_k-1)}\omega}^{(n_k-1)}(0)).$$ Since for all $n\in\mathbb{N}$ $T_{\sigma^{-n}\omega}^{(n)}(0)\in\overline{D}_{r_{\omega}},$ it follows that there is a further subsequence $n_{k_j}$ and $x\in\overline{D}_{r_{\omega}}$ such that $\lim_{j\to\infty}T_{\sigma^{-(n_{k_j}-1)}\omega}^{(n_{k_j}-1)}(0)=x,$ and so $x\in L(\omega).$ It follows that $y=T_{\omega}(x)\in T_{\omega}(L(\omega)).$ Thus, $L(\sigma\omega)=T_{\omega}(L(\omega)).$ 

    Let $\text{diam}_D$ denote the diameter in the hyperbolic distance on the disc, that is, for $A\subset D,$ $\text{diam}_D(A)=\sup_{y,z\in A}d_{D}(y,z).$ Since $\overline{D}_{r_\omega}$ is a compact subset of $D,$ it follows that $\text{diam}_D(L(\omega))\le\text{diam}_{D}(\overline{D}_{r_{\omega}})<\infty.$ 
    
    From Theorem \ref{scwarz_pick}, for all $\omega\in\Omega,$ $T_{\omega}$ is a semicontraction of the hyperbolic distance, that is, for all $y, z\in D,$ $$\begin{cases}
        d_{D}(T_{\omega}(y),T_{\omega}(z))<d_{D}(y,z),&\text{if }\deg(T_{\omega})\ge2\\d_{D}(T_{\omega}(y),T_{\omega}(z))=d_{D}(y,z),&\text{if }\deg(T_{\omega})=1.
    \end{cases}$$ In particular, for any compact $K\subset D,$ if $\deg(T_{\omega})\ge2,$ there exists $c_{K}:=c_{K}(\omega)<1$ such that for all $y,z\in K,$ $$d_{D}(T_{\omega}(y),T_{\omega}(z))\le c_{K}d_{D}(y,z).$$ If $\text{diam}_D(L(\omega))=0,$ then the claim follows. Suppose that $\text{diam}_D(L(\omega))>0.$  Since $L(\sigma\omega)=T_{\omega}(L(\omega)),$ we have \begin{align*}
        \text{diam}_D(L(\sigma\omega))&=\text{diam}_D(T_{\omega}(L(\omega))\\&=\sup_{y,z\in T_{\omega}(L(\omega))}d_{D}(y,z)\\&=\sup_{y,z\in L(\omega)}d_{D}(T_{\omega}(y), T_{\omega}(z))\\&\le\sup_{y,z\in L(\omega)}d_{D}(y,z)\\&=\text{diam}_D( L(\omega)).
\end{align*}

Since $\sigma$ is ergodic and $\mathbb{P}$-preserving, $\text{diam}_{D}(L(\omega))=\text{diam}_{D}(L(\sigma\omega))$ almost surely. Since $L(\omega)$ is closed and a subset of $\overline{D}_{r_\omega},$ $L(\omega)$ is compact (with respect to the hyperbolic distance). Thus, if $\text{diam}_{D}(L(\omega))>0,$ then $$\begin{cases}\text{diam}_{D}(L(\sigma\omega)) <\text{diam}_D(L(\omega)),&\text{if }\deg(T_{\omega})\ge2\\ \text{diam}_{D}(L(\sigma\omega))=\text{diam}_D(L(\omega)),&\text{if }\deg(T_{\omega})=1.\end{cases}$$ Since $\mathbb{P}(\omega\in\Omega:\deg(T_{\omega})\ge2)>0,$ for $\mathbb{P}$-a.e. $\omega\in\Omega,$ there is a $n:=n(\omega)$ such that $\text{diam}_{D}(L(\sigma^{n}\omega))<\text{diam}_{D}(L(\omega)),$ a contradiction.
  It follows that for $\mathbb{P}$-a.e. $\omega\in\Omega,$ $\text{diam}_{D}(L(\omega))=0,$ and so there exists $c_{\omega}\in\overline{D}_{r_{\omega}}$ such that $L(\omega)=\{c_{\omega}\}.$ 
Thus, for $\mathbb{P}$-a.e. $\omega\in\Omega,$ $\lim_{n\to\infty}T_{\sigma^{-n}\omega}^{(n)}(0)=c_{\omega}.$ Furthermore, $c_{\omega}$ is measurable, since it is the limit of measurable functions, and $$c_{\sigma\omega}=\lim_{n\to\infty}T_{\sigma^{-n}(\sigma\omega)}^{(n)}(0)=\lim_{n\to\infty}T_{\omega}(T_{\sigma^{-(n-1)}\omega}^{(n-1)}(0))=T_{\omega}(c_{\omega}).$$

\end{proof}
\begin{remark}
    The above result holds replacing $0$ with any fixed $z\in D.$ If $c_{\omega}$ is independent of the choice of $z\in D,$ then $\{c_{\omega}\}_{\omega\in\Omega}$ is a random attracting fixed point.
\end{remark}

\subsection{Characterization of eventually expanding on average cocycles}\label{section_EEOA}

In this section, we will show that if admissible $\mathcal{T}$ has a random fixed point, then it is eventually EOA.

The following theorem of Tischler gives a lower bound on the derivative of a Blaschke product $T$ on the circle.
\begin{theorem}[Theorem 2(i) \cite{tischler}]\label{tischler}
    Let $T$ be a finite Blaschke product. Let $r:=r_{T}(R)=\max_{|z|=R}|T(z)|.$ Suppose for some $R<1$ that $r<R.$ Then,  $\inf_{z\in\mathbb{T}}|T'(z)|\ge\frac{\log r}{\log R}>1.$
\end{theorem}

\begin{proposition}\label{rfp_oi_eoa}
    Let $\mathcal{T}$ be an admissible Blaschke product cocycle. Suppose that there is measurable $c:\Omega\to D$ such that for $\mathbb{P}$-a.e. $\omega\in\Omega,$ $\lim_{n\to\infty}T_{\sigma^{-n}\omega}^{(n)}(0)=c_{\omega}.$
    Then, $\mathcal{T}$ is eventually expanding on average. 
\end{proposition}

\begin{proof}  
    Let $\phi_{a}(z):=\frac{z+a}{1+\bar{a}z}.$ Let $$\tilde{T}_{\omega}:=\phi_{c_{\sigma\omega}}^{-1}\circ T_{\omega}\circ\phi_{c_{\omega}}$$ so that $\tilde{T}_{\omega}(0)=0$ and $\tilde{T}$ is measurable. Since $c_{\sigma\omega}=T_{\omega}(c_{\omega}),$ for $\mathbb{P}$-a.e. $\omega\in\Omega$ and all $n\in\mathbb{N},$ $$\tilde{T}_{\omega}^{(n)}=\phi_{c_{\sigma^{n}\omega}}^{-1}\circ T_{\omega}^{(n)}\circ\phi_{c_{\omega}}.$$ Since each map $\tilde{T}_{\omega}$ fixes the origin, for every $R<1$ and all $\omega\in\Omega,$ we have from Corollary \ref{tischler} (or Schwarz's lemma) that $\sup_{|z|=R}|\tilde{T}_{\omega}(z)|\le R.$
    Let $$\Omega_{\delta}:=\{\omega\in\Omega:\text{the zeros of }T_\omega\text{ are contained in }D_{1-\delta}\text{ and }\deg(T_{\omega})\ge2\}$$ and $$\tilde{\Omega}_{\varepsilon}:=\{\omega\in\Omega:|c_{\omega}|\le1-\varepsilon\}.$$ For sufficiently small $\delta>0,$ $\mathbb{P}(\Omega_{\delta})>0.$ By the measurability of $c_{\omega}$ and the fact that $\mathbb{P}(\omega\in\Omega:c_{\omega}\in D)=1$, there exists $\varepsilon:=\varepsilon_{\delta}>0$ such that $\mathbb{P}(\tilde{\Omega}_{\varepsilon})>1-\mathbb{P}(\Omega_{\delta}).$ It thus follows that $\mathbb{P}(\tilde{\Omega}_{\varepsilon}\cap\Omega_{\delta})>0.$
    
     Take $\omega\in\tilde{\Omega}\cap\Omega_{\delta}.$  From Proposition \ref{zero_bound_rfp}, we know that if $|c_{\omega}|<1-\varepsilon$ then there is $\tilde{R}:=\tilde{R}_{\varepsilon, \delta}<1$ such that all the zeros of $\tilde{T}_{\omega}=\phi_{c_{\sigma\omega}}^{-1}\circ T_{\omega}\circ\phi_{c_{\omega}}$ are contained in $D_{\tilde{R}}.$ Let $r:=R\left(\frac{R+\tilde{R}}{1+R\tilde{R}}\right).$ From Corollary \ref{Blaschke_bound}, we have that for every $R<1,$ $$\sup_{|z|=R}|\tilde{T}_{\omega}(z)|\le\begin{cases}
        R\left(\frac{R+\tilde{R}}{1+R\tilde{R}}\right)^{\deg(T_{\omega})-1}\le r<R,&\text{if }\omega\in\tilde{\Omega}_{\varepsilon}\cap\Omega_{\delta}\\R,&\text{if }\omega\not\in\tilde{\Omega}_{\varepsilon}\cap\Omega_{\delta}.
    \end{cases}$$ It follows from Theorem \ref{tischler} that $$\inf_{z\in\mathbb{T}}|\tilde{T}_{\omega}'(z)|\ge\begin{cases}
        \frac{\log r}{\log R}>1, & \text{if }\omega\in \tilde{\Omega}_{\varepsilon}\cap\Omega_{\delta}\\1, & \text{if }\omega\not\in \tilde{\Omega}_{\varepsilon}\cap\Omega_{\delta}.
    \end{cases}$$ 
    We thus have \begin{align*}
        \inf_{z\in\mathbb{T}}|(\tilde{T}_{\omega}^{(n)})'(z)|\ge\left(\frac{\log r}{\log R}\right)^{m_{n,\omega}},
    \end{align*} where $m_{n,\omega}=\sum_{k=1}^{n}\textbf{1}_{\tilde{\Omega}_{\varepsilon}\cap\Omega_{\delta}}(\sigma^{k}\omega).$ For $\mathbb{P}$-a.e. $\omega\in\Omega,$ it follows from Birkhoff ergodic theorem that $$\lim_{n\to\infty}\frac{m_{n,\omega}}{n}=\lim_{n\to\infty}\frac{1}{n}\sum_{k=1}^{n}\textbf{1}_{\tilde{\Omega}_{\varepsilon}\cap\Omega_{\delta}}(\sigma^{k}\omega)=\mathbb{P}(\tilde{\Omega}_{\varepsilon}\cap\Omega_{\delta})>0.$$

      Since, for $\mathbb{P}$-a.e. $\omega\in\Omega$ and all $n\in\mathbb{N},$ $T_{\omega}^{(n)}=\phi_{c_{\sigma^{n}\omega}}\circ \tilde{T}_{\omega}^{(n)}\circ\phi_{c_{\omega}}^{-1}$, we have \begin{align*}
        \frac{1}{n}\log\inf_{z\in\mathbb{T}}|(T_{\omega}^{(n)})'(z)|&\ge\frac{1}{n}\log\inf_{z\in\mathbb{T}}|\phi_{-c_{\omega}}'(z)|+\frac{1}{n}\log\inf_{z\in\mathbb{T}}|\phi_{c_{\sigma^{n}\omega}}'(z)|\\&+\frac{1}{n}\log\inf_{z\in\mathbb{T}}|(\tilde{T}_{\omega}^{(n)})'(z)|\\&\ge\frac{1}{n}\log\inf_{z\in\mathbb{T}}|\phi_{-c_{\omega}}'(z)|+\frac{1}{n}\log\inf_{z\in\mathbb{T}}|\phi_{(c_{\sigma^{n}\omega})}'(z)|\\&+\frac{m_{n,\omega}}{n}\log\left(\frac{\log r}{\log R}\right).
    \end{align*} By assumption, $\omega\in\tilde{\Omega}_{\varepsilon}$ and so, from Proposition \ref{derivative_bounds}, $$\frac{\varepsilon}{2-\varepsilon}|\le|\phi_{-c_{\omega}}'(z)|\le\frac{2-\varepsilon}{\varepsilon}$$ for all $z\in\mathbb{T}.$ Thus, the first term on the right hand side converges to zero as $n\to\infty$. Take the subsequence $l_k:=l_k(\omega)$ of hitting times to $\tilde{\Omega}_{\varepsilon}\cap\Omega_{\delta},$ that is, $\sigma^{l_{k}}\omega\in\tilde{\Omega}_{\varepsilon}\cap\Omega_{\delta}$ for all $k\in\mathbb{N}$ and $\sigma^{n}\omega\not\in \tilde{\Omega}_{\varepsilon}\cap\Omega_{\delta}$ for $l_{k}+1\le n\le l_{k+1}-1, \forall k\in\mathbb{N}.$ For $\mathbb{P}$-a.e. $\omega\in\Omega,$ $l_k\to\infty$ as $k\to\infty.$ From the fact that $T_{\omega}^{(n)}(c_{\omega})=c_{\sigma^{n}\omega}$ and $\sigma^{l_k}\omega\in\tilde{\Omega}_{\varepsilon},$  it follows that the second term converges to zero along this subsequence. Thus,  \begin{align*}\lim_{k\to\infty}\frac{1}{l_k}\log\inf_{z\in\mathbb{T}}|(T_{\omega}^{(l_k)})'(z)|&\ge\lim_{k\to\infty}\frac{m_{l_k,\omega}}{l_k}\log\left(\frac{\log r}{\log R}\right)\\&=\mathbb{P}(\tilde{\Omega}_{\varepsilon}\cap\Omega_{\delta})\log\left(\frac{\log r}{\log R}\right).\end{align*} 
    From Kingman's subadditive ergodic theorem, $\lim_{n\to\infty}\frac{1}{n}\log\inf_{z\in\mathbb{T}}|(T_{\omega}^{(n)})'(z)|$ exists almost surely and thus $\lim_{k\to\infty}\frac{1}{l_k}\log\inf_{z\in\mathbb{T}}|(T_{\omega}^{(l_k)})'(z)|$ is equal to the limit of the original sequence. Thus \begin{align*}
         \lim_{n\to\infty}\frac{1}{n}\log\inf_{z\in\mathbb{T}}|(T_{\omega}^{(n)})'(z)|\ge\mathbb{P}(\tilde{\Omega}_{\varepsilon}\cap\Omega_{\delta})\log\left(\frac{\log r}{\log R}\right)>0.
    \end{align*}

    By Kingman's subadditive ergodic theorem, we have \begin{align*}\lim_{n\to\infty}\frac{1}{n}\int_{\Omega}\log\inf_{z\in\mathbb{T}}|(T_{\omega}^{(n)})'(z)|\,d\mathbb{P}(\omega)>0\end{align*} and thus $\mathcal{T}$ is eventually EOA. 
    
\end{proof}

\begin{remark}
    Proposition  \ref{N_admissible_rfp} and \ref{rfp_oi_eoa} imply that admissible $\mathcal{T}$ has a random attracting fixed point in $D$ if and only if $\mathcal{T}$ is eventually EOA.
\end{remark}

\begin{remark}
    Proposition \ref{no_circle_pt} and \ref{rfp_oi_eoa} tell us that if admissible $\mathcal{T}$ is not eventually EOA, then $\mathbb{P}(\Omega_{\mathbb{T}})=1.$
\end{remark}

\subsection{Proof of Theorem \ref{possible_limits}}\label{section_proof_main_thm}

In this section, we give several results needed for the proof of Theorem \ref{possible_limits}, followed by the proof itself.

The following theorem is a particular case of a result of Gouëzel and Karlsson \cite[Corollary 5.2]{subadditive_MET} where all maps map the unit disc to itself. It ensures the almost sure convergence of the cocycle to a boundary attracting random fixed point if almost every orbit tends to the boundary sufficiently quickly.

\begin{theorem}[Corollary 5.2 \cite{subadditive_MET}]\label{contract_disc_conv}
    Let $d_{D}$ denote the hyperbolic distance on $D,$ given by $d_{D}(z,w)=\log\left(\frac{1+\left|\frac{z-w}{1-\bar{w}z}\right|}{1-\left|\frac{z-w}{1-\bar{w}z}\right|}\right).$ Let $(\mathcal{T}, \sigma, \mathbb{P})$ be an ergodic cocycle of holomorphic maps of $D$ such that $\int_{\Omega}d_{D}(T_{\omega}(0),0)\,d\mathbb{P}(\omega)<\infty.$ Then, unless for $\mathbb{P}$-a.e. $\omega$ $$\lim_{n\to\infty}\frac{1}{n}d_{D}(T_{\sigma^{-n}\omega}^{(n)}(z),z)=0,$$ it holds that for $\mathbb{P}$-a.e. $\omega\in\Omega,$ the orbit $T_{\sigma^{-n}\omega}^{(n)}(z)$ converges for all $z\in D$ to some boundary point $\xi_{\omega}\in\mathbb{T}$ independent of $z\in D.$
\end{theorem}

The following theorem is a particular case of a result of Benini, Evdoridou, Fagella, Rippon, and Stallard \cite[Theorem 4.1]{benini_1}, where every map is a finite Blaschke product, and so maps the unit disc to itself.
\begin{theorem}[Theorem 4.1 \cite{benini_1}]\label{holomorphic_sequence}
    Let $T_{n}:D\to D,$ $n\in\mathbb{N},$ be a sequence of Blaschke products. Suppose there exists $z_{0}\in D$ such that $$\sum_{n=0}^{\infty}(1-|T_{n}(z_{0})|)<\infty.$$ Then, for a.e. $z\in \mathbb{T},$ $$\lim_{n\to\infty}|T_{n}(z)-T_{n}(z_{0})|=0.$$
\end{theorem}

The following definition can be seen as the random analogue to the deterministic physical measure introduced by Eckmann and Ruelle \cite{eckman_ruelle}. 
\begin{definition}[(Random) physical measure]\label{physical_measure_defn}
    We will call a $\mathcal{T}^{*}$-invariant measure $\mu$ a physical measure if for $\mathbb{P}$-a.e. $\omega\in\Omega,$ there exists a set of $z\in\mathbb{T}$ with positive Lebesgue measure such that $$\lim_{n\to\infty}\frac{1}{n}\sum_{k=0}^{n-1}\delta_{T_{\sigma^{-k}\omega}^{(k)}(z)}=\mu_{\omega},$$ or equivalently, for all $f\in C(\mathbb{T}),$ $$\lim_{n\to\infty}\frac{1}{n}\sum_{k=0}^{n-1}f(T_{\sigma^{-k}\omega}^{(k)}(z))=\int_{\mathbb{T}}f\,d\mu_{\omega}.$$ 
\end{definition}
\begin{remark}
    Alternatively, some works refer to a measure as physical if it is the limit of the pushforward of the Lebesgue measure, or of any measure which is absolutely continuous with respect to Lebesgue, for example, Blumenthal and Young in \cite{blumenthal_young}. Our results also hold for this alternate definition of physical measure.
\end{remark}

We will now prove Theorem \ref{possible_limits}.
\begin{proof}[Proof of Theorem \ref{possible_limits}]

    \begin{description}
    \item[(1)]\label{case1} $\mathcal{T}$ is eventually EOA.
    The claim follows from Theorem \ref{buzzi_blaschke}, Theorem \ref{rfp_oi_eoa} and Lemma \ref{N_admissible_rfp}. 
    \item[(2)]$\mathcal{T}$ is not eventually EOA. In this case, $\mathbb{P}(\Omega_{\mathbb{T}})=1$ and $\mathbb{P}$-a.e. $\omega\in\Omega,$ $(T_{\sigma^{-n}\omega}^{(n)}(0))_{n=1}^{\infty}$ has a limit point on $\mathbb{T}.$ Thus, for $\mathbb{P}$-a.e. $\omega\in\Omega,$ there exists $x:=x(\omega)\in\mathbb{T}$ and a subsequence $n_k$ such that $\lim_{k\to\infty}T_{\sigma^{-n-k}\omega}^{(n_k)}(0)=x(\omega).$ By the maximum modulus principle, $\lim_{k\to\infty}T_{\sigma^{-n-k}\omega}^{(n_k)}(z)=x(\omega)$ for all $z\in D.$  
    
    Recall that each Blaschke product is a semicontraction of the hyperbolic metric, that is, $d_{D}(T_{\omega}(y),T_{\omega}(z))\le d_{D}(y,z)$ for all $\omega\in\Omega$ and $y,z\in D.$ Thus, \begin{align*}
        d_{D}(T_{\sigma^{-(n+m)}\omega}^{(n+m)}(0),0)&\le d_{D}(T_{\sigma^{-(n+m)}\omega}^{(n+m)}(0),T_{\sigma^{-m}\omega}^{(m)}(0))\\&+d_{D}(T_{\sigma^{-m}\omega}^{(m)}(0),0)\\&\le d_{D}(T_{\sigma^{-(n+m)}\omega}^{(n)}(0),0)+d_{D}(T_{\sigma^{-m}\omega}^{(m)}(0),0),
    \end{align*} and $f(n,\omega):=d_{D}(T_{\sigma^{-n}\omega}^{(n)}(0),0)$ is a subadditive function. From Kingman's subadditive ergodic theorem, it follows that $$A(\omega):=\lim_{n\to\infty}\frac{1}{n}\log\left(\frac{1+|T_{\sigma^{-n}\omega}^{(n)}(0)|}{1-|T_{\sigma^{-n}\omega}^{(n)}(0)|}\right)=\lim_{n\to\infty}\frac{1}{n}d_{D}(T_{\sigma^{-n}\omega}^{(n)}(0),0)$$ exists and  $A(\omega)$ is equal to constant $A\ge0$ for $\mathbb{P}$-a.e. $\omega\in\Omega.$ Suppose $A>0.$ It follows from Theorem \ref{contract_disc_conv}, taking $F_{n}(\omega):=T_{\sigma^{-n}\omega}^{(n)},$ that there exists $x:\Omega\to\mathbb{T}$ such that for all $z\in D,$ $$\lim_{n\to\infty}T_{\sigma^{-n}\omega}^{(n)}(z)=x_{\omega}\in\mathbb{T}.$$ Since $x_{\omega}$ is the limit of measurable functions, it follows that $x_{\omega}$ is measurable.
    
    It remains to show the convergence holds for $\text{Leb}$-a.e. $z\in\mathbb{T}.$ We have \begin{align*}
    d_{D}(T_{\sigma^{-n}\omega}^{(n)}(0),0)&=\log\left(\frac{1+|T_{\sigma^{-n}\omega}^{(n)}(0)|}{1-|T_{\sigma^{-n}\omega}^{(n)}(0)|}\right)\\&=\log(1+|T_{\sigma^{-n}\omega}^{(n)}(0)|)-\log(1-|T_{\sigma^{-n}\omega}^{(n)}(0)|).
\end{align*} Thus, \begin{align*}
    A&=\lim_{n\to\infty}\frac{1}{n}d_{D}(T_{\sigma^{-n}\omega}^{(n)}(0),0)\\&=\lim_{n\to\infty}\frac{1}{n}\left(\log(1+|T_{\sigma^{-n}\omega}^{(n)}(0)|)-\log(1-|T_{\sigma^{-n}\omega}^{(n)}(0)|)\right)\\&=-\lim_{n\to\infty}\frac{1}{n}\log(1-|T_{\sigma^{-n}\omega}^{(n)}(0)|),
\end{align*} since $0\le\log(1+|T_{\sigma^{-n}\omega}^{(n)}(0)|)\le\log2.$ Thus, for $\mathbb{P}$-a.e. $\omega\in\Omega,$ $$\lim_{n\to\infty}\frac{1}{n}\log(1-|T_{\sigma^{-n}\omega}^{(n)}(0)|)=-A<0,$$ so $$\lim_{n\to\infty}(1-|T_{\sigma^{-n}\omega}^{(n)}(0)|)^{1/n}=e^{-A}<1.$$ Fix some $r>0$ such that $e^{-A}<r<1.$ For $\mathbb{P}$-a.e. $\omega\in\Omega$ there exists $N:=N(\omega)\in\mathbb{N}$ such that for all $n\ge N$ we have $(1-|T_{\sigma^{-n}\omega}^{(n)}(0)|)^{1/n}\le r$ so $1-|T_{\sigma^{-n}\omega}^{(n)}(0)|\le r^n.$ Thus, we may write \begin{align*}
    \sum_{n=1}^{\infty}(1-|T_{\sigma^{-n}\omega}^{(n)}(0)|)&=\sum_{n=1}^{N}(1-|T_{\sigma^{-n}\omega}^{(n)}(0)|)\\&+\sum_{n=N+1}^{\infty}(1-|T_{\sigma^{-n}\omega}^{(n)}(0)|)\\&\le\sum_{n=1}^{N+1}(1-|T_{\sigma^{-n}\omega}^{(n)}(0)|)+\sum_{n=N+1}^{\infty}r^n.
\end{align*} Clearly, $\sum_{n=1}^{\infty}(1-|T_{\sigma^{-n}\omega}^{(n)}(0)|)<\infty.$  Since $T_{\omega}$ maps $D$ to itself for all $\omega\in\Omega,$ taking $F_{n}:=T_{\sigma^{-n}\omega}^{(n)}$ and $z_{0}=0,$ it follows from Theorem \ref{holomorphic_sequence} that $$\lim_{n\to\infty}|T_{\sigma^{-n}\omega}^{(n)}(z)-T_{\sigma^{-n}\omega}^{(n)}(0)|=0$$ for $\text{Leb}$-a.e. $z\in\mathbb{T}.$ Moreover, it follows that, for all continuous $f:\mathbb{T}\to\mathbb{C},$ $\text{Leb}$-a.e. $z\in\mathbb{T}$ and $\mathbb{P}$-a.e. $\omega\in\Omega,$ $\lim_{n\to\infty}f(T_{\sigma^{-n}\omega}^{(n)}(z))=f(x_{\omega})$ and so $\lim_{n\to\infty}\frac{1}{n}\sum_{k=0}^{n-1}f(T_{\sigma^{-k}\omega}^{(k)}(z))=f(x_{\omega}).$ Thus, for $\text{Leb}$-a.e. $z\in\mathbb{T}$ and $\mathbb{P}$-a.e. $\omega\in\Omega,$ $$\lim_{n\to\infty}\frac{1}{n}\sum_{k=0}^{n-1}\delta_{T_{\sigma^{-k}\omega}^{(k)}(z)}=\delta_{x_{\omega}}.$$ Furthermore, since the physical measure has a basin of full measure, it is the only physical measure for the cocycle $\mathcal{T}^{*}$. 
    \end{description}
\end{proof}

\subsection{Consequences and refinements}
In this section, we present several results related to Theorem \ref{possible_limits}.

\subsubsection{Further information on asymptotic behavior}

In this section, we explore further the limiting behavior of the cocycle in the interior, on the boundary, and in the exterior of the unit disc.

\begin{remark}
    Analogous results hold for the limiting behavior of $T_{\sigma^{-n}\omega}^{(n)}(z)$ for  $z\in\hat{\mathbb{C}}\backslash\overline{D},$ since, for all $z\in\hat{\mathbb{C}}$ and any finite Blaschke product $T,$ $$T(1/\bar{z})=\frac{1}{\overline{T(z)}}.$$ In particular, if $$\lim_{n\to\infty}T_{\sigma^{-n}\omega}^{(n)}(z)=x_{\omega}$$ for all $z\in D,$ then $$\lim_{n\to\infty}T_{\sigma^{-n}\omega}^{(n)}(z)=1/\bar{x}_{\omega}$$ for all $z\in\hat{\mathbb{C}}\backslash\overline{D}.$ 
    If $x_{\omega}\in\mathbb{T},$ then these limits coincide.
\end{remark}

\begin{corollary}
    Suppose $\mathcal{T}$ is eventually EOA. Then, for $\mathbb{P}$-a.e. $\omega\in\Omega$ and all $z\in D,$ $$\lim_{n\to\infty}\frac{1}{n}\sum_{k=0}^{n-1}\delta_{T_{\sigma^{-k}\omega}^{(k)}(z)}=\delta_{x_{\omega}}.$$
\end{corollary}
\begin{proof}
    From Theorem \ref{possible_limits}, there is measurable $x:\Omega\to D$ such that for $\mathbb{P}$-a.e. $\omega\in\Omega$ and all $z\in D,$ $\lim_{n\to\infty}T_{\sigma^{-n}\omega}^{(n)}(z)=x_{\omega}$ and the claim follows immediately.
\end{proof}

For a function $f$ on $\mathbb{T},$ we recall the definition of the harmonic extension of $f$ to $D.$
\begin{definition}\label{harmonic_extension}
    For $x\in D$ and $f\in L^{1}(\text{Leb}),$ let $\hat{f}:D\to\mathbb{C}$ be given by $$\hat{f}(x)=\int_{\mathbb{T}}f(z)\frac{1-|x|^{2}}{|x-z|^{2}}\,dm(z).$$ We call $\hat{f}$ the harmonic extension of $f$ to the unit disc.
\end{definition}

\begin{lemma}[{\cite[Theorem 1.2(b)]{Mane1991}}]\label{harmonic_limit}
    Given $f\in L^1(\text{Leb}),$ if $f$ is continuous at $w\in\mathbb{T},$ then $\lim_{z\to w}\hat{f}(z)=f(w).$
\end{lemma}

%\begin{remark}
    %A single Blaschke product $T$ is eventually expanding, that is, there exists $n\in\mathbb{N}$ such that $\inf_{z\in\mathbb{T}}|(T^n)'(z)|>0$ if and only if $T$ has a fixed point in $D.$ However, an admissible Blaschke product cocycle $\mathcal{T},$ made up of eventually expanding Blaschke products, can be constructed so that $\mathcal{T}$ is not eventually EOA.
%\end{remark}

We recall the following result for finite Blaschke products.
\begin{theorem}\label{poisson_kernel}
    Let $P_{x}(z)=\frac{1-|x|^2}{|z-x|^2}, z\in\mathbb{T}$ be the Poisson kernel associated to $x\in D.$ For a finite Blaschke product $T$ and $f\in L^{1}(\text{Leb}),$ $$\int_{\mathbb{T}}f\circ T\cdot P_{x}\,d\text{Leb}=\int_{\mathbb{T}}f\cdot P_{T(x)}\,d\text{Leb}.$$
\end{theorem}

In the following corollary, we show that if $\mathcal{T}$ has a random fixed point on $\mathbb{T}$ (starting from inside $D$), then not only does $$\lim_{n\to\infty}(T_{\sigma^{-n}\omega}^{(n)})_{*}(\text{Leb})=\delta_{x_{\omega}}$$ for $\mathbb{P}$-a.e. $\omega\in\Omega,$ but $$\lim_{n\to\infty}(T_{\sigma^{-n}\omega}^{(n)})_{*}(\text{Leb}|_{W_{\varepsilon}})=\delta_{x_{\omega}},$$ where $W_{\varepsilon}\subset\mathbb{T}$ is an $\varepsilon$-neighborhood of $x_{\omega}\in\mathbb{T}.$
\begin{corollary}\label{neighborhood_convergence_Leb}
    Suppose $\mathcal{T}$ has an attracting random fixed point on $\mathbb{T},$ that is, there exists measurable $c:\Omega\to\mathbb{T}$ such that, for $\mathbb{P}$-a.e. $\omega\in\Omega$ and all $z\in D,$ $$\lim_{n\to\infty}T_{\sigma^{-n}\omega}^{(n)}(z)=c_{\omega}.$$ For $\varepsilon>0,$ let $W_{\varepsilon}:=W_{\varepsilon}(\omega)=\{z\in\mathbb{T}:|z-x_{\omega}|<\varepsilon\}.$ Then, for $\mathbb{P}$-a.e. $\omega\in\Omega$ and all $\varepsilon>0,$ $$\lim_{n\to\infty}\text{Leb}\{z\in\mathbb{T}:T_{\sigma^{-n}\omega}^{(n)}(z)\in W_{\varepsilon}\}=1.$$
\end{corollary}
\begin{proof}
    Let $1_{W_{\varepsilon}}$ denote the characteristic function of $W_{\varepsilon}.$ Clearly $1_{W_{\varepsilon}}\in L^1(\text{Leb}).$ We have \begin{align*}
        \int_{\mathbb{T}}1_{W_{\varepsilon}}\circ T_{\sigma^{-n}\omega}^{(n)}\,d\text{Leb}&=\int_{\mathbb{T}}1_{W_{\varepsilon}}\cdot P_{T_{\sigma^{-n}\omega}^{(n)}(0)}\,d\text{Leb}\\&=\hat{1}_{W_{\varepsilon}}(T_{\sigma^{-n}\omega}^{(n)}(0)),
    \end{align*} where $\hat{1}_{W_{\varepsilon}}$ is the harmonic extension of $1_{W_{\varepsilon}}$ as in Definition \ref{harmonic_extension}. Since $1_{W_{\varepsilon}}$ is continuous at $x_{\omega},$ it follows from Lemma \ref{harmonic_limit} that $\lim_{n\to\infty}\hat{1}_{W_{\varepsilon}}(T_{\sigma^{-n}\omega}^{(n)}(0))=1_{W_{\varepsilon}}(x_{\omega})=1.$ Thus, $$\lim_{n\to\infty}\text{Leb}(\{z\in\mathbb{T}:T_{\sigma^{-n}\omega}^{(n)}(z)\in W_{\varepsilon}\})=\lim_{n\to\infty}\int_{\mathbb{T}}1_{W_{\varepsilon}}\circ T_{\sigma^{-n}\omega}^{(n)}\,d\text{Leb}=1.$$
\end{proof}

\subsubsection{Pullback-pushforward equivalence} In this section, we establish an equivalence between the boundary convergence rates of the pullback and pushforward cocycles, and show that this leads to synchronization on the circle under the hypotheses of Theorem \ref{possible_limits}(2).

The following result allows us to compare the boundary convergence rates of pullback and pushforward trajectories of the cocycle.
\begin{proposition}\label{push_pull_equiv}
    For $\mathbb{P}$-a.e. $\omega\in\Omega$, $\lim_{n\to\infty}\frac{1}{n}\log(1-|T_{\sigma^{-n}\omega}^{(n)}(0)|)=\lim_{n\to\infty}\frac{1}{n}\log(1-|T_{\omega}^{(n)}(0)|).$
\end{proposition}
\begin{proof}
    Let $f_{n}(\omega):=\log(1-|T_{\sigma^{-n}\omega}^{(n)}(0)|)$ and $g_{n}(\omega):=\log(1-|T_{\omega}^{(n)}(0)|).$ Recall that $f_{n}$ and $g_n$ are subadditive. By Kingman's subadditive ergodic theorem, for $\mathbb{P}$-a.e. $\omega\in\Omega,$ $\lim_{n\to\infty}\frac{1}{n}f_{n}(\omega)$ and $\lim_{n\to\infty}\frac{1}{n}g_{n}(\omega)$ exist, are a.s. equal to constant functions $A, B$, and that for $\mathbb{P}$-a.e. $\omega\in\Omega,$ $$A=\lim_{n\to\infty}\frac{1}{n}\int_{\Omega}f_{n}(\omega)\,d\pw$$ and $$B=\lim_{n\to\infty}\frac{1}{n}\int_{\Omega}g_{n}(\omega)\,d\pw.$$ 
    Notice that $f_{n}(\omega)=g_{n}(\sigma^{n}\omega).$ Since $\sigma$ is measure preserving, for all $n\in\mathbb{N},$ we must have $\int_{\Omega}f_{n}(\omega)\,d\mathbb{P}(\omega)=\int_{\Omega}f_{n}(\sigma^{n}\omega)\,d\mathbb{P}(\omega),$ and so $\int_{\Omega}g_{n}(\omega)\,d\mathbb{P}(\omega)=\int_{\Omega}f_{n}(\omega)\,d\mathbb{P}(\omega).$ Thus, the claim follows.
\end{proof}

The following corollary shows that under the hypotheses of Theorem \ref{possible_limits}\ref{case2}, we see  synchronization of almost every trajectory in forward time.
\begin{corollary}\label{push_pull_equiv_limit}
    Suppose $\mathcal{T}$ satisfies the hypotheses of Theorem \ref{possible_limits}\ref{case2}. Then, for $\text{Leb}$-a.e. $y, z\in\mathbb{T}$ and $\mathbb{P}$-a.e. $\omega\in\Omega,$ $$\lim_{n\to\infty}|T_{\omega}^{(n)}(z)-T_{\omega}^{(n)}(y)|=0.$$
\end{corollary}
\begin{proof}
    It follows from the proof of Theorem \ref{possible_limits} and Proposition \ref{push_pull_equiv} that $$\lim_{n\to\infty}\frac{1}{n}d_{D}(T_{\omega}^{(n)}(0),0)=A>0.$$ Following the proof of Theorem \ref{possible_limits}, replacing $T_{\sigma^{-n}\omega}^{(n)}(0)$ with $T_{\omega}^{(n)}(0),$ it follows that $\sum_{n=1}^{\infty}(1-|T_{\omega}^{(n)}(0)|)<\infty.$ We must have, for $\text{Leb}$-a.e. $y,z\in\mathbb{T}$ and $\mathbb{P}$-a.e. $\omega\in\Omega,$ $$\lim_{n\to\infty}|T_{\omega}^{(n)}(z)-T_{\omega}^{(n)}(y)|\le\lim_{n\to\infty}|T_{\omega}^{(n)}(z)-T_{\omega}^{(n)}(0)|+\lim_{n\to\infty}|T_{\omega}^{(n)}(y)-T_{\omega}^{(n)}(0)|=0.$$ 
\end{proof}

\subsubsection{Entropy for eventually EOA cocycles}
In this section, we show that if an admissible Blaschke product cocycle is eventually EOA, then it has positive measure theoretic entropy with respect to its unique absolute continuous random invariant measure given in Theorem \ref{possible_limits}(1).

\begin{proposition}\label{EOA_positive_entropy}
    Suppose $\mathcal{T}$ is eventually EOA, that is, $$\lim_{n\to\infty}\int_{\Omega}\frac{1}{n}\log\inf_{z\in\mathbb{T}}|(T_{\omega}^{(n)})'(z)|\,d\mathbb{P}(\omega)>0.$$ Let $\mu$ be as in Lemma \ref{N_admissible_rfp}. Then $$h_{\mu}^{fib}(\mathcal{T})=\int_{\Omega}\int_{\mathbb{T}}\log|T_{\omega}'(z)|P_{x_{\omega}}(z)\,d\text{Leb}(z)\,d\mathbb{P}(\omega)>0.$$
\end{proposition}
\begin{proof}
    First we will show that for all $n\in\mathbb{N}$ $$\frac{1}{n}\int_{\Omega}\int_{\mathbb{T}}\log|(T_{\omega}^{(n)})'(z)|P_{x_{\omega}}(z)\,d\text{Leb}(z)\,d\mathbb{P}(\omega)=\int_{\Omega}\int_{\mathbb{T}}\log|T_{\omega}'(z)|P_{x_{\omega}}(z)\,d\text{Leb}\,d\mathbb{P}(\omega).$$ Indeed, we have \begin{align*}
        &\frac{1}{n}\int_{\Omega}\int_{\mathbb{T}}\log|(T_{\omega}^{(n)})'(z)|P_{x_{\omega}}(z)\,d\text{Leb}(z)\,d\mathbb{P}(\omega)\\&=\frac{1}{n}\sum_{k=0}^{n-1}\int_{\Omega}\int_{\mathbb{T}}\log|T_{\sigma^{k}\omega}'(T_{\omega}^{(k)}(z))|P_{x_{\omega}}(z)\,d\text{Leb}(z)\,d\mathbb{P}(\omega)\\&=\frac{1}{n}\sum_{k=0}^{n-1}\int_{\Omega}\int_{\mathbb{T}}\log|T_{\sigma^{k}\omega}'(z)|P_{T_{\omega}^{(k)}(x_{\omega})}(z)d\text{Leb}(z)\,d\mathbb{P}(\omega)\\&=\frac{1}{n}\sum_{k=0}^{n-1}\int_{\Omega}\int_{\mathbb{T}}\log|T_{\sigma^{k}\omega}'(z)|P_{x_{\sigma^{k}\omega}}(z)d\text{Leb}(z)\,d\mathbb{P}(\omega)\\&=\int_{\Omega}\int_{\mathbb{T}}\log|T_{\omega}'(z)|P_{x_{\omega}}(z)\,d\text{Leb}(z)\,d\mathbb{P}(\omega),
    \end{align*} where the third line follows from Theorem \ref{poisson_kernel} and the last line from $\sigma$ being $\mathbb{P}$-preserving. Thus, for all $n\in\mathbb{N},$ \begin{align*}
        \int_{\Omega}\int_{\mathbb{T}}\log|T_{\omega}'(z)|P_{x_{\omega}}(z)\,d\text{Leb}(z)\,d\mathbb{P}(\omega)&=\frac{1}{n}\int_{\Omega}\int_{\mathbb{T}}\log|(T_{\omega}^{(n)})'(z)|P_{x_{\omega}}(z)\,d\text{Leb}(z)\,d\mathbb{P}(\omega)\\&\ge\frac{1}{n}\int_{\Omega}\log\inf_{z\in\mathbb{T}}|(T_{\omega}^{(n)})'(z)|\int_{\mathbb{T}}P_{x_{\omega}}\,d\text{Leb}(z)\,d\mathbb{P}(\omega)\\&=\frac{1}{n}\int_{\Omega}\log\inf_{z\in\mathbb{T}}|(T_{\omega}^{(n)})'(z)|\,d\mathbb{P}(\omega).
    \end{align*} It follows that \begin{align*}
        h_{\mu}^{fib}(\mathcal{T})&=\int_{\Omega}\int_{\mathbb{T}}\log|T_{\omega}'(z)|P_{x_{\omega}}(z)\,d\text{Leb}(z)\,d\mathbb{P}(\omega)\\&\ge\lim_{n\to\infty}\int_{\Omega}\frac{1}{n}\log\inf_{z\in\mathbb{T}}|(T_{\omega}^{(n)})'(z)|\,d\mathbb{P}(\omega)\\&>0.
    \end{align*} 
\end{proof}

\section{A transition between synchronization and chaos}\label{example_section}
In this section, we investigate in detail a 2-map family of admissible Blaschke product cocycles, $\mathcal{T}_{a,b}$,
\begin{equation*}
T_0(z)=\left(\frac{z+a}{1+az}\right)^2, T_1(z)=\left(\frac{z+b}{1+bz}\right)^2 \tag{$*$} \label{eq:star}
\end{equation*}
where $0\le a<1/3<b<1$, with ergodic, invertible, $\mathbb{P}$-preserving driving $\sigma$. We take these values of $a,b$ so that $T_0$ has an attracting fixed point inside $D$ and $1$ is a repelling fixed point, while $T_1$ has an attracting fixed point at $1$. In addition, $T_i([0,1])\subseteq[0,1],$ $i=0,1.$ Let $p:=\mathbb{P}(\omega\in\Omega:T_{\omega}=T_0)$ and $p_{a,b}:=\frac{\log\left(\frac{1+b}{2(1-b)}\right)}{\log\left(\frac{(1-a)(1+b)}{(1+a)(1-b)}\right)}.$ Figure \ref{fig:interval_map} shows an example of one such pair (viewed as maps of the interval instead of maps of $\mathbb{T}$).

\begin{figure}\label{interval_example}
    \centering
    \includegraphics[width=0.5\linewidth]{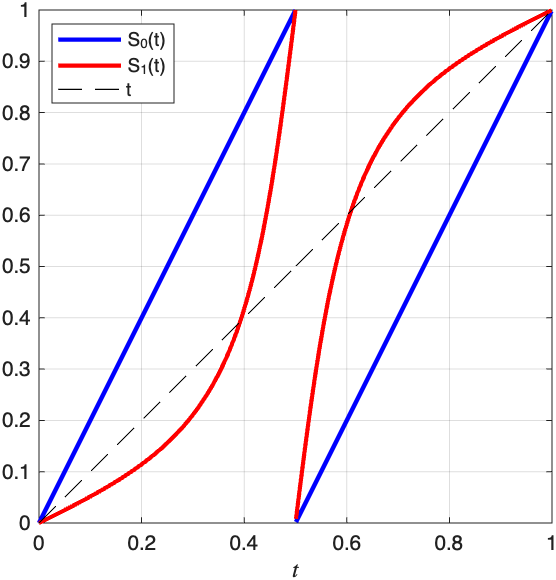}
    \caption{$t$ vs $S_0(t)=(Q^{-1}T_0Q)(t),$ $S_1=(Q^{-1}T_1Q)(t)$ where $Q(t)=e^{2\pi i t},$ $T_0(z)=z^2,$ $T_1(z)=\left(\frac{z+3/5}{1+3z/5}\right)^2$}
    \label{fig:interval_map}
\end{figure}

\begin{proposition}\label{example_expansion}
    If $p>p_{a,b}$ then $\mathcal{T}_{a,b}$ is EOA, otherwise $\mathcal{T}$ is not eventually EOA.
\end{proposition}
\begin{proof}
    Note that $\inf_{z\in\mathbb{T}}|T_{0}'(z)|=|T_{0}'(1)|=\frac{2(1-a)}{1+a},$ $\inf_{z\in\mathbb{T}}|T_{1}'(z)|=|T_{1}'(1)|=\frac{2(1-b)}{1+b}.$ Thus, \begin{align*}\lim_{n\to\infty}\int_{\Omega}\frac{1}{n}\log\inf_{z\in\mathbb{T}}|(T_{\omega}^{(n)})'(z)|\,d\mathbb{P}(\omega)&=\int_{\Omega}\log|T_{\omega}'(1)|\,d\mathbb{P}(\omega)\\&=p\log |T_0'(1)|+(1-p)\log|T_1'(1)|\\&=\log2+p\log\left(\frac{1-a}{1+b}\right)+(1-p)\log\left(\frac{1-b}{1+b}\right).\end{align*} Rearranging gives $$\begin{cases}\int_{\Omega}\log|T_{\omega}'(1)|\,d\mathbb{P}(\omega)>0,& p>p_{a,b}\\\int_{\Omega}\log|T_{\omega}'(1)|\,d\mathbb{P}(\omega)<0,&p< p_{a,b}\\\int_{\Omega}\log|T_{\omega}'(1)|\,d\mathbb{P}(\omega)=0,&p= p_{a,b}.\end{cases}$$
\end{proof}

We give the main result of this section in the following theorem.
\begin{theorem}\label{theorem_example}
    For $\mathcal{T}_{a,b},$ one of the following 3 cases hold:
    \begin{description}
        \item[$p>p_{a,b}$] There exists measurable $x:\Omega\to[0,1]$ such that, for $\mathbb{P}$-a.e. $\omega\in\Omega$, $\lim_{n\to\infty}T_{\sigma^{-n}\omega}^{(n)}(z)=x_{\omega}<1.$ $\mu$ as given in Theorem \ref{buzzi_blaschke} is the unique physical random measure for $\mathcal{T}_{a,b}^{*}.$ 
        \item[$p< p_{a,b}$] For $\mathbb{P}$-a.e. $\omega\in\Omega,$ all $z\in D$ and $\text{Leb}$-a.e. $z\in\mathbb{T},$ $\lim_{n\to\infty}T_{\sigma^{-n}\omega}^{(n)}(z)=1$ and $\lim_{n\to\infty}T_{\omega}^{(n)}(z)=1.$ Moreover, $\mu:=\{\mu_{\omega}\}_{\omega\in\Omega},$ given by $\mu_{\omega}=\delta_{1},$ is the unique random physical measure for $\mathcal{T}_{a,b}^{*}.$ 
        \item[$p=p_{a,b}$] For $\mathbb{P}$-a.e. $\omega\in\Omega$ and all $z\in D,$ $\lim_{n\to\infty}T_{\sigma^{-n}\omega}^{(n)}(z)=1.$ Moreover, for $\mathbb{P}$-a.e. $\omega\in\Omega$ and any $\varepsilon>0,$ taking $W_{\varepsilon}\subset\mathbb{T}$ to be an $\varepsilon$-neighborhood of $1,$ $\lim_{n\to\infty}(T_{\sigma^{-n}\omega}^{(n)})_{*}(\text{Leb}|_{W_{\varepsilon}})=\delta_1.$ If the maps are chosen iid, then $\lim_{n\to\infty}T_{\omega}^{(n)}(z)\neq1$ for all $z\in D$ and $\text{Leb}$-a.e. $z\in\mathbb{T}.$
    \end{description}
\end{theorem}
The proof of Theorem \ref{theorem_example} will be shown in several sections.

\subsection{Proof of Theorem \ref{theorem_example}(1)}

\begin{proof}[$p>p_{a,b}$]
    In this case, from Proposition \ref{example_expansion}, $\mathcal{T}$ is EOA. It follows from Theorem \ref{possible_limits}\ref{case1} (in fact, from Theorem \ref{buzzi_blaschke}) that there exists measurable $x:\Omega\to D$ such that, for $\mathbb{P}$-a.e. $\omega\in\Omega$, $\lim_{n\to\infty}T_{\sigma^{-n}\omega}^{(n)}(z)=x_{\omega}.$ Since $T_{\omega}([0,1))\subset[0,1)$ for all $\omega\in\Omega,$ $x_{\omega}\in[0,1)$ for $\mathbb{P}$-a.e. $\omega\in\Omega.$ It follows from Theorem \ref{buzzi_blaschke} that $\mu$ is the unique $\mathcal{T}_{a,b}$-invariant physical measure for $\mathcal{T}_{a,b}.$
\end{proof}

\subsection{Proof of Theorem \ref{theorem_example}(2)}
\begin{proof}[$p<p_{a,b}$]
    We will show that $\lim_{n\to\infty}\frac{1}{n}\log(1-|T_{\omega}^{(n)}(0)|)<0$ for $\mathbb{P}$-a.e. $\omega\in\Omega,$ and thus Theorem \ref{possible_limits}$(2)$ implies, for $\mathbb{P}$-a.e. $\omega\in\Omega,$ all $z\in D$ and $\text{Leb}$-a.e. $z\in\mathbb{T},$ $\lim_{n\to\infty}T_{\sigma^{-n}\omega}^{(n)}(z)=1.$ It can be shown that if $0\le z<1$ and $0\le a<1$ then \begin{align*}d_{D}\left(0,\left(\frac{z+a}{1+az}\right)^2\right)&=\log\left(\frac{1+\left(\frac{z+a}{1+az}\right)^2}{1-\left(\frac{z+a}{1+az}\right)^2}\right)\\&\ge\log\left(\frac{1+z}{1-z}\right)-\log\left(\frac{2(1-a)}{1+a}\right)\\&=d_{D}(0,z)-\log\left(\frac{2(1-a)}{1+a}\right).\end{align*} Thus, $d_{D}(0,T_{\omega}(z))\ge d_{D}(0,z)-\log|T_{\omega}'(1)|$ and \begin{align*}
        \frac{1}{n}d_{D}(0,T_{\omega}^{(n)}(0))&\ge-\log|(T_{\omega}^{(n)})'(1)|.
    \end{align*} From Birkhoff's ergodic theorem, it follows that for $\mathbb{P}$-a.e. $\omega\in\Omega,$ \begin{align*}\lim_{n\to\infty}\frac{1}{n}\log(1-|T_{\omega}^{(n)}(0)|)&=-\lim_{n\to\infty}\frac{1}{n}d_{D}(0, T_{\omega}^{(n)}(0))\\&\le\frac{1}{n}\sum_{k=0}^{n-1}\log|T_{\sigma^{k}\omega}'(1)|\\&=\int_{\Omega}\log|T_{\omega}'(1)|d\mathbb{P}(\omega)<0.\end{align*} From Proposition \ref{push_pull_equiv}, it holds that for $\mathbb{P}$-a.e. $\omega\in\Omega,$ $$\lim_{n\to\infty}\frac{1}{n}\log(1-|T_{\sigma^{-n}\omega}^{(n)}(0)|)<0.$$ Thus, from Theorem \ref{possible_limits}\ref{case2}, for $\mathbb{P}$-a.e. $\omega\in\Omega$ and $\text{Leb}$-a.e. $z\in\mathbb{T},$ $\lim_{n\to\infty}T_{\sigma^{-n}\omega}^{(n)}(z)=1.$ Moreover, from Corollary \ref{push_pull_equiv_limit} and the maximum modulous principle, it follows that $\lim_{n\to\infty}T_{\omega}^{(n)}(z)=1$ for $\mathbb{P}$-a.e. $\omega\in\Omega,$ all $z\in D$ and $\text{Leb}$-a.e. $z\in\mathbb{T}.$ Thus, $\mu:=\{\mu_{\omega}\}_{\omega\in\Omega},$ given by $\mu_{\omega}=\delta_{1},$ is not only a $\mathcal{T}^{*}$-invariant measure, but the unique physical measure for $\mathcal{T}^{*}.$
\end{proof}

\subsection{Proof of Theorem \ref{theorem_example}(3)}

The proof of Theorem \ref{theorem_example}(3) will be given in several parts.

In the following proposition, we describe the limiting behavior of pullback orbits starting in the disc.
\begin{proposition}\label{example_limit_pts_disc}
Suppose $p=p_{a,b}.$ For $\mathbb{P}$-a.e. $\omega\in\Omega$ and all $z\in D,$ $\lim_{n\to\infty}T_{\sigma^{-n}\omega}^{(n)}(z)=1.$

\end{proposition}
\begin{proof}

        Note that the only possible limit point of $T_{\sigma^{-n}\omega}^{(n)}(z)$ on the circle for $z\in[0,1)$ is $1.$ Since $\mathcal{T}$ is not eventually EOA, it follows that $\mathbb{P}(\Omega_{\mathbb{T}})=1.$ Thus, for $\mathbb{P}$-a.e. $\omega\in\Omega$ and all $z\in D,$ there exists $n_{k}:=n_k(\omega)$ such that $T_{\sigma^{-n_k}\omega}^{(n_k)}(0)$ converges to 1 and so, by the maximum modulus principle, $\lim_{k\to\infty}T_{\sigma^{-n_k}\omega}^{(n_k)}(z)=1$ for all $z\in D.$

        Fix $\varepsilon>0.$ Let $n_{\varepsilon}:=n_{\varepsilon}(\omega)=\inf\{j\in\mathbb{N}:T_{\sigma^{-j}\omega}^{(j)}(0)>1-\varepsilon\}.$ $n_{\varepsilon}$ is finite for $\mathbb{P}$-a.e. $\omega\in\Omega.$ It follows that, for all $n\ge n_{\varepsilon},$ \begin{align*}
         T_{\sigma^{-n}\omega}^{(n)}(0)&=T_{\sigma^{-n_{\varepsilon}}\omega}^{(n_{\varepsilon})}(T_{\sigma^{-n}\omega}^{(n-n_{\varepsilon})}(0))\\&\ge T_{\sigma^{-n_{\varepsilon}}\omega}^{(n_{\varepsilon})}(0)\\&>1-\varepsilon,
     \end{align*} where the second line follows from the fact that $T_{\omega}(z)$ restricted to $[0,1]$ is an increasing function for all $\omega\in\Omega.$ Thus, for $\mathbb{P}$-a.e. $\omega\in\Omega,$ $\lim_{n\to\infty}T_{\sigma^{-n}\omega}^{(n)}(0)=1$ and by maximum modulus principle, it follows that  $\lim_{n\to\infty}T_{\sigma^{-n}\omega}^{(n)}(z)=1$ for all $z\in D.$ 
\end{proof}

The following corollary follows immediately from Proposition \ref{example_limit_pts_disc} and \ref{neighborhood_convergence_Leb}.
\begin{corollary}
    Suppose $p=p_{a,b}.$ For $\varepsilon>0,$ let $W_{\varepsilon}=\{z\in\mathbb{T}:|z-1|<\varepsilon\}.$ Then, for all $\varepsilon>0$ and $\mathbb{P}$-a.e. $\omega\in\Omega,$ $$\lim_{n\to\infty}(T_{\sigma^{-n}\omega}^{(n)})_{*}(\text{Leb}|_{W_{\varepsilon}})=\delta_1.$$
\end{corollary}

In the following proposition, we give several useful estimates. We will prove the remaining claims in Theorem \ref{theorem_example} by conjugating the maps to the upper half plane.
\begin{proposition}\label{plane_bounds}
    Let $\tilde{T}_{a}:\overline{\mathbb{H}}\to\overline{\mathbb{H}}$ for $0<a<1$ be given by $\tilde{T}_a(x)=\frac{2(1-a^2)x}{(1+a)^2-(1-a)^2x^2},$ where $\mathbb{H}$ denotes the upper half plane. Then, for $0\le x<1,$ $$|\tilde{T}_a(x)|\ge|\tilde{T}_a'(0)||x|.$$ 
    
    For $x\in[0,1]i,$ $$|\tilde{T}_a(x)|\le|\tilde{T}_a'(0)||x|.$$ 
    
    Fix $0<\varepsilon<1.$ For $x\in (0,\varepsilon)i,$ and $\frac{4a(1-a^2)+2(1-a)^3\varepsilon^2}{(1+a)^3+(1-a)^2(1+a)\varepsilon^2}\le\delta\le|\tilde{T}_{a}'(0)|,$ $$|\tilde{T}_{a}(x)|\ge\left(|T_{a}'(0)|-\delta\right)|x|.$$
\end{proposition}
\begin{proof}
    Trivially, $\tilde{T}_a(x)|\ge|\tilde{T}_a'(0)||x|$ holds for $x=0.$ For $x\in(0,1)$ we have
    \begin{align*}
        &|\tilde{T}_{a}(x)|\ge|T_{a}'(1)||x|\\&\iff\frac{2(1-a^2)|x|}{|(1+a)^2-(1-a)^2x^2|}\ge\frac{2(1-a)|x|}{1+a}\\&\iff \frac{(1+a)^2}{|(1+a)^2-(1-a)^2x^2|}\ge1\\&\iff|x|\le\sqrt{\frac{2(1+a)^2}{(1-a)^2}}.
    \end{align*} We have, for $x=yi, y\in(0,1],$ \begin{align*}
       &|\tilde{T}_a(x)|\le|\tilde{T}_{a}'(0)||x|\\&\iff \frac{2(1-a^2)|x|}{|(1+a)^2-(1-a)^2x^2|}\le \frac{2(1-a)|x|}{1+a}\\&\iff\frac{(1+a)^2}{(1+a)^2+(1-a)^2y^2}\le1
   \end{align*} which holds since $(1+a)^2+(1-a)^2y^2\ge(1+a)^2$ for $y\in(0,1]$ and $a\in[0,1).$ 
   
   Fix $\varepsilon>0.$ We want $\delta>0$ such that for $x=yi, y\in(0, \varepsilon),$ 
   \begin{equation*}
   \tilde{T}_a(x)=\frac{2(1-a^2)y}{(1+a)^2+(1-a)^2y^2}\le\left(\frac{2(1-a)}{(1+a)}-\delta\right)y=(|\tilde{T}_{a}'(0)|-\delta)|x| \tag{$*$}\label{eq_ast}.
   \end{equation*}
   Rearranging, if $\delta\ge\frac{4a(1-a^2)+2(1-a)^3y^2}{(1+a)^3+(1-a)^2(1+a)y^2}=:f_{a}(y),$ then $(*)$ holds. For $0<a<1,$ $f_a(y)\ge0$ and is an increasing function on $y>0.$ Thus, the claim holds by taking $y=\varepsilon.$
\end{proof}

In the following proposition, we show that, in contrast to the case where $p<p_{a,b},$ when $p=p_{a,b}$ forward trajectories $T_{\omega}^{(n)}(z)$ for $z\in D$ return to and escape from a small neighborhood of $1$ infinitely often.
\begin{proposition}\label{forward_divergence}
    Suppose $p=p_{a,b}$ and the maps are chosen iid. Then, for $\mathbb{P}$-a.e. $\omega\in\Omega,$ $\lim_{n\to\infty}T_{\omega}^{(n)}(z)\neq1$ for all $z\in D.$ In particular, for $\mathbb{P}$-a.e. $\omega\in\Omega,$ all $z\in D$ and sufficiently small $\varepsilon>0,$ $$\lim_{n\to\infty}\#\{0\le j\le n-1:|T_{\omega}^{(j)}(z)-1|<\varepsilon\}=\infty,$$ yet   $$\lim_{n\to\infty}\#\{0\le j\le n-1:|T_{\omega}^{(j)}(z)-1|\ge\varepsilon\}=\infty.$$ 
\end{proposition}
\begin{proof}
    Instead of $T_{\omega},$ we will use $\tilde{T}_{\omega}:=\psi\circ T_{\omega}\circ\psi^{-1},$ where $\psi:\overline{D}\to\overline{\mathbb{H}},$ $\mathbb{H}$ is the upper half plane, and $\psi$ is a homeomorphism given by $\psi(z)=i\frac{1-z}{1+z}.$ Note that $\psi(1)=0$ and so $T_{\omega}^{(n)}(z)\to1$ for $z\in\overline{D}$ if and only if $\tilde{T}_{\omega}^{(n)}(x)\to0$ for $x\in\overline{\mathbb{H}}.$  Let $T_{a}(z)=-\left(\frac{z-a}{1-az}\right)^2.$ Then $\tilde{T}_{a}(x)=\frac{2(1-a^2)x}{(1+a)^2-(1-a)^2x^2}$ and $|\tilde{T}_{a}'(0
     )|=\frac{2(1-a)}{1+a}=|T_{a}'(1)|.$

     First, note that $|\tilde{T}_{\omega}'(0)|=|T_{\omega}'(1)|$ for all $\omega\in\Omega,$ and that $\log|\tilde{T}_{\sigma^{N}\omega}'(0)|,...\log|\tilde{T}_{\sigma^{n}\omega}'(0)|$ is a sequence of iid random variables with zero mean and positive variance given by $$\tilde{\sigma}=p_{a,b}(1-p_{a,b})(\log|\tilde{T}_{0}'(0)|-\log|\tilde{T}_{1}'(0)|)^2>0.$$ From the Law of the Iterated Logarithm (LIL), $$\limsup_{n\to\infty}\frac{\log|(\tilde{T}^{(n)}_{\sigma^{-n}\omega})'(0)|}{\sqrt{2\tilde{\sigma}n\log\log n}}=1$$ and $$\liminf_{n\to\infty}\frac{\log|(\tilde{T}^{(n)}_{\sigma^{-n}\omega})'(0)|}{\sqrt{2\tilde{\sigma}n\log\log n}}=-1$$ for $\mathbb{P}$-a.e. $\omega\in\Omega$.
     Let $k_n:=k_n(\omega)=\#\{1\le j\le n:\tilde{T}_{\sigma^{-j}\omega}=\tilde{T}_0\}.$ In particular, the LIL tells us that, for $\mathbb{P}$-a.e. $\omega\in\Omega$ and any $k\in\mathbb{N},$ there exists $n,m\in\mathbb{N}$ such that $k_n\ge k$ and $m-k_m\ge k.$

      From Proposition \ref{plane_bounds}, for all $x\in[0,1]i,$ $$|\tilde{T}_{\omega}^{(n)}(x)|\le|(\tilde{T}_{\omega}^{(n)})'(0)||x|\le|(\tilde{T}_{\omega}^{(n)})'(0)|.$$ Fix $\varepsilon>0.$ It follows from the LIL that, for $\mathbb{P}$-a.e. $\omega\in\Omega,$ there exists $n\in\mathbb{N}$ such that $|(\tilde{T}_{\omega}^{(n)})'(0)|<\varepsilon$ and thus $|\tilde{T}_{\omega}^{(n)}(x)|<\varepsilon.$ Thus, for $\mathbb{P}$-a.e. $\omega\in\Omega,$ there is a subsequence $n_k:=n_k(\omega)$ such that $\lim_{k\to\infty}\tilde{T}_{\omega}^{(n_k)}(x)=0$ for $x\in(0,1]i$ and thus $\lim_{k\to\infty}T_{\omega}^{(n_k)}(z)=1$ for $z\in[0,1)$ and, from the maximum modulus principle, for all $z\in D.$

     From Proposition \ref{plane_bounds}, for each $\varepsilon>0,$ there exists $\delta_0, \delta_1>0$ such that for $x\in(0,\varepsilon)i,$ $|\tilde{T}_i(x)|\ge|\tilde{T}_{i}'(x)-\delta_i||x|, i=1,2.$ Choose $0<\tilde{\varepsilon}<\varepsilon$ such that $|\tilde{T}_{0}'(x)-\delta_0|>1.$  For all $x\in(0,\tilde{\varepsilon})i,$ we have \begin{align*}
        |\tilde{T_{\omega}}(x)|\ge\begin{cases}
            |\tilde{T}_{0}'(0)-\delta_0||x|,&\text{if }\tilde{T}_{\omega}=\tilde{T}_0\\|\tilde{T}_{1}'(0)-\delta_1||x|,&\text{if }\tilde{T}_{\omega}=\tilde{T}_{1}.
        \end{cases}
    \end{align*}

    Suppose there exists $N:=N_{\tilde{\varepsilon}}\in\mathbb{N}$ such that for all $n\ge N,$  $|\tilde{T}_{\omega}^{(n)}(x)|<\tilde{\varepsilon}$ for $x=i$ (and thus for all $x\in(0,1]i$). We have $$|\tilde{T}_{\omega}^{(n)}(x)|\ge|\tilde{T}_{0}'(0)-\delta_0|^{k_n}|\tilde{T}_{1}'(0)-\delta_1|^{n-N-k_n}|T_{\omega}^{(N)}(x)|.$$ We can find $n$ such that $k_n$ is arbitrarily large. Thus, there exists $n>N$ such that $$\tilde{T}_{\omega}^{(n)}(z)|\ge|\tilde{T}_{0}'(0)-\delta|^{k_n}|\tilde{T}_{1}'(0)-\delta|^{n-N-k_n}|T_{\omega}^{(N)}(x)|\ge\tilde{\varepsilon},$$ a contradiction. Thus, for $\mathbb{P}-a.e.$ $\omega\in\Omega,$ $|\tilde{T}_{\omega}^{(n)}(x)|\ge\tilde{\varepsilon}$ and  $\lim_{n\to\infty}\tilde{T}_{\omega}^{(n)}(x)\neq0.$ It follows that $\lim_{n\to\infty}T_{\omega}^{(n)}(z)\neq1$ for $z\in[0,1)$ and thus, by the maximum modulus principle, $\lim_{n\to\infty}T_{\omega}^{(n)}(z)\neq1$ for all $z\in D.$ In particular, for $\mathbb{P}$-a.e. $\omega\in\Omega$ and sufficiently small $\varepsilon>0,$ $$\lim_{n\to\infty}\#\{0\le j\le n-1:|T_{\omega}^{(j)}(z)-1|<\varepsilon\}=\infty,$$ yet   $$\lim_{n\to\infty}\#\{0\le j\le n-1:|T_{\omega}^{(j)}(z)-1|\ge\varepsilon\}=\infty.$$ 
\end{proof}

The proof of the following proposition follows similarly to the proof of Proposition \ref{forward_divergence}.
\begin{proposition}
    Suppose $p=p_{a,b}$ and the maps are chosen iid. Then, for $\mathbb{P}$-a.e. $\omega\in\Omega,$ $\lim_{n\to\infty}T_{\omega}^{(n)}(z)\neq1$ for $\text{Leb}$-a.e. $z\in\mathbb{T}.$
\end{proposition}
\begin{proof}
    We will show that $\lim_{n\to\infty}T_{\omega}^{(n)}(z)\neq1$ for $\mathbb{P}$-a.e. $\omega\in\Omega$ and $\text{Leb}$-a.e. $z\in\mathbb{T}.$ 
    
     Note that $T_{\omega}^{(n)}(z)\to1$ for $\text{Leb}$-a.e. $z\in\mathbb{T}$ if and only if $\tilde{T}_{\omega}^{(n)}(x)\to0$ for $\text{Leb}$-a.e. $x\in\overline{\mathbb{R}}.$

       Thus, for $|y|\le M(a):=\sqrt{\frac{2(1+a)^2}{(1-a)^2}},$
      $$|\tilde{T}_{\omega}(y)|\ge\begin{cases}
        |\tilde{T}_{0}'(1)||y|,&\text{if }\tilde{T}_{\omega}=\tilde{T}_{0}\\|\tilde{T}_{1}'(1)||y|,&\text{if }\tilde{T}_{\omega}=\tilde{T}_{1}.
    \end{cases}$$ Suppose for the sake of contradiction that $\tilde{T}_{\omega}^{(n)}(x)\to0$ for $\mathbb{P}$-a.e. $\omega\in\Omega$ and a positive measure set of $x\in\overline{\mathbb{R}}.$  Fix $0<\varepsilon<M(a)$ and choose $(\omega, x)$  such that there exists $N:=N(\omega,x)$ where for all $n\ge N,$ $0<|\tilde{T}_{\omega}^{(n)}(x)|<\varepsilon,$ that is, $|\tilde{T}_{\omega}^{(n)}(x)|<\varepsilon$ and $x\not\in\{(\tilde{T}_{\omega}^{(n)})^{-1}(0):n\in\mathbb{N}\}.$ Thus, for $n\ge N,$ \begin{align*}|\tilde{T}_{\omega}^{(n)}(x)|&=|\tilde{T}_{\sigma^{N}\omega}^{(n-N)}(\tilde{T}_{\omega}^{(N)}(x))|\\&\ge|\tilde{T}_{0}'(0)|^{k_n}|\tilde{T}_{1}'(0)|^{n-N-k_n}|\tilde{T}_{\omega}^{(N)}(x)|,\end{align*} where $k_n=\#\{N\le i\le n:\tilde{T}_{\sigma^{i}\omega}=\tilde{T}_{0}\}.$ Thus, \begin{align*}\log|\tilde{T}_{\omega}^{(n)}(x)|\ge &k_n\log|\tilde{T}_{0}'(0)|+\\&+(n-N-k_n)\log|\tilde{T}_{1}'(0)|+\log|T_{\omega}^{(N)}(x)|,\end{align*} where the last term is independent of $n.$ It follows from the LIL that, for $\mathbb{P}$-a.e. $\omega\in\Omega,$ $$\limsup_{n\to\infty}\frac{k_n\log|\tilde{T}_{0}'(0)|+(n-N-k_n)\log|\tilde{T}_{1}'(0)|}{\sqrt{2\tilde{\sigma}n\log\log n}}=1.$$ Thus, for $\mathbb{P}$-a.e. $\omega\in\Omega,$ there exists $m\ge N$ such that $$|\tilde{T}_{0}'(0)|^{k_m}|\tilde{T}_{1}'(0)|^{m-N-k_m}|\ge\frac{\varepsilon}{|\tilde{T}_{\omega}^{(N)}(x)|},$$ and thus $$|\tilde{T}_{\omega}^{(m)}(x)|\ge|\tilde{T}_{0}'(0)|^{k_m}|\tilde{T}_{1}'(0)|^{m-N-k_m}|\tilde{T}_{\omega}^{(N)}(x)|\ge\varepsilon,$$ a contradiction. Thus $\tilde{T}_{\omega}^{(n)}(x)\not\to0$ and so $T_{\omega}^{(n)}(z)\not\to1.$ 
\end{proof}

\begin{remark}
    In the deterministic case, if $T$ of degree at least 2 has a parabolic fixed point $z^{*}\in\mathbb{T}$ ($|T'(z^{*})|=1$), then $z^{*}$ attracts all points in the disc, a contrast to the random example, where $1$ does not attract the forward orbit of points in the disc. In the parabolic deterministic case, the orbit is dense for $\text{Leb}$-a.e. $z\in\mathbb{T},$ and in particular $\lim_{n\to\infty}T^n(z)\neq z^{*}$ for $\text{Leb}$-a.e. $z\in\mathbb{T},$ if and only if $T''(z^{*})=0,$ so $T$ is somewhat degenerate at the fixed point \cite{baker_domains2}. If $T''(z^{*})\neq0,$ then $\lim_{n\to\infty}T^n(z)=z^{*}$ for $\text{Leb}$-a.e. $z\in\mathbb{T}.$ 
\end{remark}

There is a transition from zero to positive measure theoretic entropy for this family at the critical $p_{a,b}$ value. 
\begin{corollary}
    Let $\mu$ be the absolutely continuous invariant measure for $\mathcal{T}_{a,b}^{*}$ for $p>p_{a,b}$ and $\mu$ the Dirac measure at 1 otherwise.  Then $$\begin{cases}h_{\mu}^{fib}(\mathcal{T})>0,&p>p_{a,b}\\h_{\mu}^{fib}(\mathcal{T})=0,&p\le p_{a,b}.
        
    \end{cases}$$

\end{corollary}
\begin{proof}
    The result follows from Proposition \ref{EOA_positive_entropy} and the definition of measure theoretic entropy.
\end{proof}

\appendix
\section{Blaschke product estimates}
We require several estimates relating to Blaschke products. 

\begin{lemma}\label{Möbius bound}
     Let $\phi_{a}(z):=\frac{z+a}{1+\bar{a}z},$ $a\in D.$ For $z\in D,$ $$\left|\frac{|z|-|a|}{1-|a||z|}\right|\le|\phi_{a}(z)|\le\frac{|z|+|a|}{1+|a||z|}.$$
\end{lemma}
\begin{proof}
        First, we show that $|\phi_{a}(z)|\le\frac{|z|+|a|}{1+|a||z|}$ for $z\in D.$ We have
    \begin{align*}
        |\phi_a(z)|^2&=\frac{(z+a)(\bar{z}+\bar{a})}{(1+\bar{a}z)(1+a\bar{z})}\\&=\frac{|z|^2+|a|^2+2\text{Re}(\bar{a}z)}{1+|a|^2|z|^2+2\text{Re}(\bar{a}z)}.
    \end{align*} Thus, \begin{align*}|&\phi_a(z)|\le\frac{|z|+|a|}{1+|a||z|}\\&\iff \frac{|z|^2+|a|^2+2\text{Re}(\bar{a}z)}{1+|a|^2|z|^2+2\text{Re}(\bar{a}z)}\le\left(\frac{|z|+|a|}{1+|a||z|}\right)^2\\&\iff(|z|^2+|a|^2+2\text{Re}(\bar{a}z))(1+|a||z|)^2\le(1+|a|^2|z|^2+2 \text{Re}(\bar{a}z))(|z|+|a|)^2\\&\iff 0\le 2(1-|a|^2)(1-|z|^2)(|az|-\text{Re}(\bar{a}z)).\end{align*} Clearly, all the terms on the right hand side are nonnegative and so the claim follows.

    Now we show that $|\phi_a(z)|\ge\left|\frac{|z|-|a|}{1-|az|}\right|$ for $z\in D.$ We have
    \begin{align*}
        &|\phi_a(z)|\ge\left|\frac{|z|-|a|}{1-|a||z|}\right|\\&\iff \frac{|z|^2+|a|^2+2\text{Re}(\bar{a}z)}{1+|a|^2|z|^2+2\text{Re}(\bar{a}z)}\ge\left(\frac{|z|-|a|}{1-|a||z|}\right)^2\\&\iff(|z|^2+|a|^2+2\text{Re}(\bar{a}z))(1-|a||z|)^2\ge(1+|a|^2|z|^2+2 \text{Re}(\bar{a}z))(|z|-|a|)^2\\&\iff 0\le 2(1-|a|^2)(1-|z|^2)(|az|+\text{Re}(\bar{a}z)),
    \end{align*} and the claim follows.
\end{proof}
\begin{corollary}\label{Blaschke_bound}
    For a Blaschke product $T(z)=\theta\prod_{i=1}^{n}\frac{z-a_{i}}{1-\bar{a}_iz},$ Lemma \ref{Möbius bound} gives $\prod_{i=1}^{n}\left|\frac{|z|-|a_i|}{1-|a_i||z|}\right|\le|T(z)|\le\prod_{i=1}^{n}\frac{|z|+|a_i|}{1+|a_i||z|}.$
\end{corollary}

In the following proposition, we compute bounds on the zeroes of a finite Blachke product pre- and post-composed with Möbius transformations.
\begin{proposition}\label{zero_bound}
     Suppose the zeroes of a Blaschke product $T$ of degree at most $n$ are contained in $D_{1-\delta}.$ Then, the zeroes of $\phi_{b}^{-1}\circ T\circ\phi_{a},$ with $a,b\in D,$  are contained in $\tilde{R}:=\tilde{R}_{a,b,\delta,n}=\frac{1+|a|+|b|^{1/n}(1+|a|(1-\delta))-\delta}{1+|a|+|b|^{1/n}(1+|a|(1-\delta))-|a|\delta}<1.$
\end{proposition}
\begin{proof}
    
    The zeroes of $T\circ\phi_{a}$ are $\phi_{a}^{-1}(z_{i})=\phi_{-a}(z_{i}),$ $i=1,...,n$ where $z_{i}$ is the $i^{th}$ zero of $T.$ Thus from Lemma \ref{Möbius bound}, we have \begin{align*}
        |\phi_{-a}(z_i)|&\le\frac{|z_{i}|+|a|}{1+|a||z_i|}\\&\le\frac{\sup_i|z_i|+|a|}{1+|a|\sup_i|z_i|}\\&<\frac{1+|a|-\delta}{1+|a|-|a|\delta}=:c_{a,\delta},
    \end{align*} where the second line follows from $\frac{|z_{i}|+|a|}{1+|a||z_i|}$ being an increasing function of $|z_i|.$ Thus, the zeroes of $g:=T\circ\phi_a$ are contained in $D_{c_{a,\delta}}.$ 

    The zeroes of $\phi_{b}^{-1}\circ g$ are the solutions to $g(z)=b.$ We show that the solutions to $g(z)=b$ lie in $D_{\tilde{R}_{a,b,\delta,n}}$ where $\tilde{R}:=\tilde{R}_{a,b,\delta, n}=\frac{c_{a,\delta}+|b|^{1/n}}{1+c_{a,\delta}|b|^{1/n}}.$ Indeed, by Rouché's theorem, if, for some $r<1,$ $|g(z)|>|b|$ for all $|z|=r,$ then $g(z)=b$ has $n$ solutions in $D_{r},$ and so no solutions in $D\backslash D_{r}.$ Thus, $\phi_b^{-1}\circ T\circ\phi_a$ has all of its zeroes in $D_r.$ It remains to show that taking $r=\tilde{R}_{a,b,\delta, n}$ satisfies $|g(z)|>|b|$ for all $|z|=r.$ Take $r>c_{a,\delta}.$ From Corollary \ref{Blaschke_bound}, for all $|z|=r,$ $$|(T\circ\phi_a)(z)|>\left(\frac{r-c_{a,\delta}}{1-rc_{a,\delta}}\right)^n.$$ Thus, choosing $r$ such that $\left(\frac{r-c_{a,\delta}}{1-rc_{a,\delta}}\right)^n=|b|$ ensures that all solutions to $g(z)=b$ lie in $D_r$ and thus that all the zeroes of $\phi_{b}^{-1}\circ T\circ\phi_{a}$ are contained in $D_r$. Rearranging, we obtain $$r=\frac{|b|^{1/n}+c_{a,\delta}}{1+c_{a,\delta}|b|^{1/n}}$$ and the claim follows from substituting $c_{a,\delta}$ and simplifying.
\end{proof}

\begin{proposition}\label{zero_bound_rfp}
     Suppose the zeroes of a Blaschke product $T$ of degree $n$ are contained in $D_{1-\delta}.$ Then, the zeroes of $\phi_{T(a)}^{-1}\circ T\circ\phi_{a},$ with $a\in D,$  are contained in $\hat{R}:=\hat{R}_{a,\delta}=1-\frac{\delta(1-|a|)}{(2-\delta)(1+|a|)}<1.$
\end{proposition}
\begin{proof}
    From Proposition \ref{zero_bound}, we know the zeroes are contained in $D_{\tilde{R}_{a,\delta,n}},$ where \begin{align*}\tilde{R}_{a,\delta,n}&=\frac{1+|a|+|T(a)|^{1/n}(1+|a|(1-\delta))-\delta}{1+|a|+|T(a)|^{1/n}(1+|a|(1-\delta))-|a|\delta}\\&=1-\frac{\delta(1-|a|)}{1+|a|+|T(a)|^{1/n}(1+|a|(1-\delta))-|a|\delta}.\end{align*}
 Using  $\deg(T)=n$ and Lemma \ref{Möbius bound}, we have $$|T(a)|\le\left(\frac{|a|+1-\delta}{1+|a|(1-\delta)}\right)^n.$$ Thus, \begin{align*}
    \tilde{R}_{a,\delta,n}&=1-\frac{\delta(1-|a|)}{1+|a|+|T(a)|^{1/n}(1+|a|(1-\delta))-|a|\delta}\\&\le1-\frac{\delta(1-|a|)}{1+|a|+\left(\left(\frac{|a|+1-\delta}{1+|a|(1-\delta)}\right)^n\right)^{1/n}(1+|a|(1-\delta))-|a|\delta}\\&=1-\frac{\delta(1-|a|)}{2+2|a|-\delta-|a|\delta}\\&=1-\frac{\delta(1-|a|)}{(2-\delta)(1+|a|)}=:\hat{R}_{a,\delta}.
\end{align*}
\end{proof}

The following proposition follows from a result of Martin \cite[Proposition 1]{martin}. We include a proof for completeness.
\begin{proposition}\label{derivative_bounds}
    Let $T(z)=\rho\prod_{i=1}^{d}\frac{z-a_{i}}{1-\bar{a_i}z}, d\ge1$ be a finite Blaschke product. Then, for all $z\in\mathbb{T},$ $$\sum_{i=1}^{d}\frac{1-|a_i|}{1+|a_i|}\le|T'(z)|\le\sum_{i=1}^{d}\frac{1+|a_i|}{1-|a_i|}.$$
\end{proposition}
\begin{proof}
    For all $z\in\hat{\mathbb{C}}\backslash\{T^{-1}(0)\},$ $$\frac{T'(z)}{T(z)}=\sum_{i=1}^{d}\frac{1-|a_i|^2}{(z-a_i)(1-\bar{a_i}z)}.$$ Take $z\in\mathbb{T}.$ We have \begin{align*}
        \frac{zT'(z)}{T(z)}&=z\sum_{i=1}^{d}\frac{1-|a_i|^2}{(z-a_i)(1-\bar{a}_iz)}\\&=\sum_{i=1}^{d}\frac{1-|a_i|^2}{\bar{z}(z-a_i)(1-\bar{a}_iz)}\\&=\sum_{i=1}^{d}\frac{1-|a_i|^2}{(z-a_i)(\bar{z}-\bar{a}_i)}\\&=\sum_{i=1}^{d}\frac{1-|a_i|^2}{|z-a_i|^2},
    \end{align*} where the second line follows from $z=1/\bar{z}$ for $z\in\mathbb{T}.$ Clearly, for $z\in\mathbb{T},$ $|zT'(z)/T(z)|=|T'(z)|$ and thus $$|T'(z)|=\sum_{i=1}^{n}\frac{1-|a_i|^2}{|z-a_i|^2}.$$ The result follows from $1-|a_i|\le|z-a_i|\le1+|a_i|$.
\end{proof}

\section{Admissibility of $\mathcal{T}^{(N)}$}\label{N_admissible}
\begin{lemma}
    Fix $N\in\mathbb{N}.$ Suppose $\mathcal{T}$ is admissible. Then $\mathcal{T}^{(N)}:=(T_{\omega}^{(N)})_{\omega\in\Omega}$ is admissible.
\end{lemma}
\begin{proof} We will show that $(1)$-$(4)$ in Definition \ref{admissible} hold.
\begin{description}
    \item[(1)] Since $\deg(T_{\omega}^{(N)})=\prod_{k=0}^{N-1}\deg(T_{\sigma^{k}\omega}),$ it follows that $\deg(T_{\omega}^{(N)})\ge\deg(T_{\omega})$ for all $N\in\mathbb{N}$ and $\omega\in\Omega.$ Thus, $$\mathbb{P}(\omega:\deg(T_{\omega}^{(N)})\ge2)\ge\mathbb{P}(\omega:\deg(T_{\omega})\ge2)>0.$$

    \item[(2)]
    Recall that $\deg(T_{\omega}^{(N)})=\prod_{k=0}^{N-1}\deg(T_{\sigma^{k}\omega}).$ Thus, by the chain rule, $$\frac{\deg(T_{\omega}^{(N)})}{\inf_{z\in\mathbb{T}}|(T_{\omega}^{(N)})'(z)|}\le\prod_{k=0}^{N-1}\frac{\deg(T_{\sigma^{k}\omega})}{\inf_{z\in\mathbb{T}}|T_{\sigma^{k}\omega}'(z)|}.$$ It follows that \begin{align*}
    \int_{\Omega}\log^{+}\frac{\deg(T_{\omega}^{(N)})}{\inf_{z\in\mathbb{T}}|(T_{\omega}^{(N)})'(z)|}\,d\mathbb{P}(\omega)&\le\int_{\Omega}\sum_{k=0}^{N-1}\log^{+}\frac{\deg(T_{\sigma^{k}\omega})}{\inf_{z\in\mathbb{T}}|T_{\sigma^{k}\omega}'(z)|}\,d\mathbb{P}(\omega)\\&=\sum_{k=0}^{N-1}\int_{\Omega}\log^{+}\frac{\deg(T_{\omega})}{\inf_{z\in\mathbb{T}}|T_{\omega}'(z)|}\,d\mathbb{P}(\omega)\\&=N\int_{\Omega}\log^{+}\frac{\deg(T_{\omega})}{\inf_{z\in\mathbb{T}}|T_{\omega}'(z)|}\,d\mathbb{P}(\omega)\\&<\infty,
    \end{align*} where the second line follows from $\sigma$ being $\mathbb{P}$-preserving.

    \item[(3)] By repeated application of chain rule and product rule, we have \begin{align*}
        \frac{(T_{\omega}^{(N)})''(z)}{((T_{\omega}^{(N)})'(z))^2}&=\sum_{k=0}^{N-1}\frac{T_{\sigma^{k}\omega}''(T_{\omega}^{(k-1)}(z))}{((T_{\sigma^{k}\omega}'(T_{\omega}^{(k-1)}(z))^2}\prod_{j=k}^{N-1}\frac{1}{T_{\sigma^{j}\omega}'(T_{\omega}^{(j-1)}(z))},
    \end{align*} so \begin{align*}
        \int_{\mathbb{T}}\left|\frac{(T_{\omega}^{(N)})''(z)}{((T_{\omega}^{(N)})'(z))^2}\right|\,d\text{Leb}(z)&\le\sum_{k=0}^{N-1}\int_{\mathbb{T}}\left|\frac{T_{\sigma^{k}\omega}''(T_{\omega}^{(k-1)}(z))}{((T_{\sigma^{k}\omega}'(T_{\omega}^{(k-1)}(z))^2}\prod_{j=k}^{N-1}\frac{1}{T_{\sigma^{j}\omega}'(T_{\omega}^{(j-1)}(z))}\right|d\text{Leb(z)}\\&\le\sum_{k=0}^{N-1}\frac{1}{\inf_{z\in\mathbb{T}}|(T_{\sigma^{k}\omega}^{(N-k)})'(z)|}\int_{\mathbb{T}}\left|\frac{T_{\sigma^{k}\omega}''(T_{\omega}^{(k-1)}(z))}{((T_{\sigma^{k}\omega}'(T_{\omega}^{(k-1)}(z))^2}\right|d\text{Leb(z)}.
    \end{align*} Let $I_{j,k}$ be the $j^{th}$ branch of $T_{\omega}^{(k-1)}$ for $j=1,...,\deg(T_{\omega}^{(k-1)}).$ Through a change of variables, we have \begin{align*} \int_{\mathbb{T}}\frac{T_{\sigma^{k}\omega}''(T_{\omega}^{(k-1)}(z))}{((T_{\sigma^{k}\omega}'(T_{\omega}^{(k-1)}(z))^2}d\text{Leb(z)}&=\sum_{j=1}^{\deg(T_{\omega}^{(k)})-1}\int_{I_{j,k}}\left|\frac{T_{\sigma^{k}\omega}''(z)}{(T_{\sigma^{k}\omega}'(z))^2}\frac{1}{(T_{\omega}^{(k)})'(z)}\right|d\text{Leb}(z)\\&\le\frac{\deg(T_{\omega}^{(k)})}{\inf_{z\in\mathbb{T}}|(T_{\omega}^{(k)})'(z)|}\int_{\mathbb{T}}\left|\frac{T_{\sigma^{k}\omega}''(z)}{(T_{\sigma^{k}\omega}'(z))^2}\right|d\text{Leb}(z),
    \end{align*} and so \begin{align*}
        &\int_{\Omega}\log^{+}\int_{\mathbb{T}}\left|\frac{(T_{\omega}^{(N)})''(z)}{((T_{\omega}^{(N)})'(z))^2}\right|\,d\text{Leb}(z)\,d\mathbb{P}(\omega)\\&\le\int_{\Omega}\log^{+}\left(\sum_{k=0}^{N-1}\frac{1}{\inf_{z\in\mathbb{T}}|(T_{\sigma^{k}\omega}^{(N-k)})'(z)|}\int_{\mathbb{T}}\left|\frac{T_{\sigma^{k}\omega}''(T_{\omega}^{(k-1)}(z))}{((T_{\sigma^{k}\omega}'(T_{\omega}^{(k-1)}(z))^2}\right|d\text{Leb(z)}\right)\,d\mathbb{P}(\omega)\\&\le \int_{\Omega}\log^{+}\left(\sum_{k=0}^{N-1}\frac{1}{\inf_{z\in\mathbb{T}}|(T_{\sigma^{k}\omega}^{(N-k)})'(z)|}\frac{\deg(T_{\omega}^{(k)})}{\inf_{z\in\mathbb{T}}|(T_{\omega}^{(k)})'(z)|}\int_{\mathbb{T}}\left|\frac{T_{\sigma^{k}\omega}''(z)}{(T_{\sigma^{k}\omega}'(z))^2}\right|d\text{Leb}(z)\right)\,d\mathbb{P}(\omega).\end{align*} 
        
        Using $\deg(T_{\omega})\ge1, \deg(T_{\omega}^{(k)})=\prod_{i=0}^{k-1}\deg(T_{\sigma^{i}\omega})$ and $\inf_{z\in\mathbb{T}}|(T_{\omega}^{(k)})'(z)|\ge\prod_{i=0}^{k-1}\inf_{z\in\mathbb{T}}|T_{\sigma^{i}\omega}'(z)|,$ we have \begin{align*}
        &\int_{\Omega}\log^{+}\left(\sum_{k=0}^{N-1}\frac{1}{\inf_{z\in\mathbb{T}}|(T_{\sigma^{k}\omega}^{(n-k)})'(z)|}\frac{\deg(T_{\omega}^{(k)})}{\inf_{z\in\mathbb{T}}|(T_{\omega}^{(k)})'(z)|}\int_{\mathbb{T}}\left|\frac{T_{\sigma^{k}\omega}''(z)}{(T_{\sigma^{k}\omega}'(z))^2}\right|d\text{Leb}(z)\right)\,d\mathbb{P}(\omega)
        \\&\le\int_{\Omega}\log^{+}\left(\prod_{k=0}^{N-1}\frac{\deg(T_{\sigma^{k}\omega})}{\inf_{z\in\mathbb{T}}|T_{\sigma^{k}\omega}'(z)|}\sum_{k=0}^{N-1}\int_{\mathbb{T}}\left|\frac{T_{\sigma^{k}\omega}''(z)}{(T_{\sigma^{k}\omega}'(z))^2}\right|d\text{Leb}(z)\right)\,d\mathbb{P}(\omega)
        \\&\le\int_{\Omega}\sum_{k=0}^{N-1}\log\frac{\deg(T_{\sigma^{k}\omega})}{\inf_{z\in\mathbb{T}}|T_{\sigma^{k}\omega}'(z)|}d\mathbb{P}(\omega)+
        \\&+\int_{\Omega}\log^{+}\left(n\max_{0\le k\le N-1}\int_{\mathbb{T}}\left|\frac{T_{\sigma^{k}\omega}''(z)}{(T_{\sigma^{k}\omega}'(z))^2}\right|d\text{Leb}(z)\right)\,d\mathbb{P}(\omega)
        \\&=\log N+\sum_{k=0}^{N-1}\int_{\Omega}\log\frac{\deg(T_{\sigma^{k}\omega})}{\inf_{z\in\mathbb{T}}|T_{\sigma^{k}\omega}'(z)|}d\mathbb{P}(\omega)
        \end{align*}\begin{align*}&\hspace{\fill}+\int_{\Omega}\log^{+}\left(\max_{0\le k\le N-1}\int_{\mathbb{T}}\left|\frac{T_{\sigma^{k}\omega}''(z)}{(T_{\sigma^{k}\omega}'(z))^2}\right|d\text{Leb}(z)\right)\,d\mathbb{P}(\omega)
        \\&\le\log N+N\int_{\Omega}\log\frac{\deg(T_{\omega})}{\inf_{z\in\mathbb{T}}|T_{\omega}'(z)|}d\mathbb{P}(\omega)+\sum_{k=0}^{N-1}\int_{\Omega}\log^{+}\left(\int_{\mathbb{T}}\left|\frac{T_{\sigma^{k}\omega}''(z)}{(T_{\sigma^{k}\omega}'(z))^2}\right|d\text{Leb}(z)\right)d\mathbb{P}(\omega)
        \\&=\log N +N\left(\int_{\Omega}\log^{+}\frac{\deg(T_{\omega})}{\inf_{z\in\mathbb{T}}|T_{\omega}'(z)|}d\mathbb{P}(\omega)+\int_{\Omega}\log^{+}\left(\int_{\mathbb{T}}\left|\frac{T_{\omega}''(z)}{(T_{\omega}'(z))^2}\right|d\text{Leb}(z)\right)d\mathbb{P}(\omega)\right)\\&<\infty,
    \end{align*} where the last two lines follow from $\sigma$ being $\mathbb{P}$-preserving and $\mathcal{T}$ admissible.

    \item[(4)] Since $$\log\inf_{z\in\mathbb{T}}|(T_{\omega}^{(N)})'(z)|\ge\log\prod_{k=0}^{N-1}\inf_{z\in\mathbb{T}}|T_{\sigma^{k}\omega}'(z)|=\sum_{k=0}^{N-1}\log\inf_{z\in\mathbb{T}}|T_{\sigma^{k}\omega}'(z)|,$$ we have \begin{align*}\int_{\Omega}\log\inf_{z\in\mathbb{T}}|(T_{\omega}^{(N)})'(z)|\,d\mathbb{P}(\omega)&\ge\sum_{k=0}^{N-1}\int_{\Omega}\log\inf_{z\in\mathbb{T}}|T_{\sigma^{k}\omega}'(z)|\,d\mathbb{P}(\omega)\\&=N\int_{\Omega}\log\inf_{z\in\mathbb{T}}|T_{\omega}'(z)|\,d\mathbb{P}(\omega)\\&>-\infty,\end{align*}\end{description} since $\sigma$ is $\mathbb{P}$-preserving.\end{proof}

\printbibliography

@article {expandPujals,
    AUTHOR = {Pujals, Enrique R. and Robert, Leonel and Shub, Michael},
     TITLE = {Expanding maps of the circle rerevisited: positive {L}yapunov
              exponents in a rich family},
   JOURNAL = {Ergodic Theory Dynam. Systems},
  FJOURNAL = {Ergodic Theory and Dynamical Systems},
    VOLUME = {26},
      YEAR = {2006},
    NUMBER = {6},
     PAGES = {1931--1937},
      ISSN = {0143-3857},
   MRCLASS = {37E10 (37A35)},
  MRNUMBER = {2279272},
MRREVIEWER = {Wenxiang Sun},
       DOI = {10.1017/S0143385706000368},
       URL = {https://doi.org/10.1017/S0143385706000368},
}

@article {AFGV2,
    AUTHOR = {Atnip, Jason and Froyland, Gary and Gonz\'{a}lez-Tokman, Cecilia
              and Vaienti, Sandro},
     TITLE = {Thermodynamic formalism for random weighted covering systems},
   JOURNAL = {Comm. Math. Phys.},
  FJOURNAL = {Communications in Mathematical Physics},
    VOLUME = {386},
      YEAR = {2021},
    NUMBER = {2},
     PAGES = {819--902},
      ISSN = {0010-3616},
   MRCLASS = {37D35 (37E05 37H12)},
  MRNUMBER = {4294282},
MRREVIEWER = {Eugen Mihailescu},
       DOI = {10.1007/s00220-021-04156-1},
       URL = {https://doi.org/10.1007/s00220-021-04156-1},
}

@article {buzzi_SRB,
    AUTHOR = {Buzzi, J\'{e}r\^{o}me},
     TITLE = {Absolutely continuous {S}.{R}.{B}. measures for random
              {L}asota-{Y}orke maps},
   JOURNAL = {Trans. Amer. Math. Soc.},
  FJOURNAL = {Transactions of the American Mathematical Society},
    VOLUME = {352},
      YEAR = {2000},
    NUMBER = {7},
     PAGES = {3289--3303},
      ISSN = {0002-9947},
   MRCLASS = {37D20 (37A25 37E05 37H99)},
  MRNUMBER = {1707698},
MRREVIEWER = {Bernard Schmitt},
       DOI = {10.1090/S0002-9947-00-02607-6},
       URL = {https://doi.org/10.1090/S0002-9947-00-02607-6},
}

@article {buzzi_edoc,
    AUTHOR = {Buzzi, J\'{e}r\^{o}me},
     TITLE = {Exponential decay of correlations for random {L}asota-{Y}orke
              maps},
   JOURNAL = {Comm. Math. Phys.},
  FJOURNAL = {Communications in Mathematical Physics},
    VOLUME = {208},
      YEAR = {1999},
    NUMBER = {1},
     PAGES = {25--54},
      ISSN = {0010-3616},
   MRCLASS = {37H15 (37H05)},
  MRNUMBER = {1729876},
MRREVIEWER = {Bernard Schmitt},
       DOI = {10.1007/s002200050746},
       URL = {https://doi.org/10.1007/s002200050746},
}

@article {iid_Mobius,
    AUTHOR = {Karmakar, Satyajit and Key, Eric S.},
     TITLE = {Compositions of random {M}\"{o}bius transformations},
   JOURNAL = {Stochastic Anal. Appl.},
  FJOURNAL = {Stochastic Analysis and Applications},
    VOLUME = {22},
      YEAR = {2004},
    NUMBER = {3},
     PAGES = {525--557},
      ISSN = {0736-2994,1532-9356},
   MRCLASS = {60G99 (60F05)},
  MRNUMBER = {2047267},
MRREVIEWER = {V.\ Thangaraj},
       DOI = {10.1081/SAP-120030445},
       URL = {https://doi.org/10.1081/SAP-120030445},
}

@article {Mobius_plane,
    AUTHOR = {Ambroladze, Amiran and Wallin, Hans},
     TITLE = {Random iteration of {M}\"{o}bius transformations and
              {F}urstenberg's theorem},
   JOURNAL = {Ergodic Theory Dynam. Systems},
  FJOURNAL = {Ergodic Theory and Dynamical Systems},
    VOLUME = {20},
      YEAR = {2000},
    NUMBER = {4},
     PAGES = {953--962},
      ISSN = {0143-3857,1469-4417},
   MRCLASS = {37H99 (60B15 60J10)},
  MRNUMBER = {1779388},
MRREVIEWER = {G\"{o}ran\ H\"{o}gn\"{a}s},
       DOI = {10.1017/S0143385700000535},
       URL = {https://doi.org/10.1017/S0143385700000535},
}

@article {Mobius_composition,
    AUTHOR = {Jacques, Matthew and Short, Ian},
     TITLE = {Repeated compositions of {M}\"{o}bius transformations},
   JOURNAL = {Ergodic Theory Dynam. Systems},
  FJOURNAL = {Ergodic Theory and Dynamical Systems},
    VOLUME = {40},
      YEAR = {2020},
    NUMBER = {2},
     PAGES = {437--452},
      ISSN = {0143-3857,1469-4417},
   MRCLASS = {37F05 (28A80)},
  MRNUMBER = {4048300},
MRREVIEWER = {Tarakanta\ Nayak},
       DOI = {10.1017/etds.2018.43},
       URL = {https://doi.org/10.1017/etds.2018.43},
}

@article {tischler,
    AUTHOR = {Tischler, David},
     TITLE = {Blaschke products and expanding maps of the circle},
   JOURNAL = {Proc. Amer. Math. Soc.},
  FJOURNAL = {Proceedings of the American Mathematical Society},
    VOLUME = {128},
      YEAR = {2000},
    NUMBER = {2},
     PAGES = {621--622},
      ISSN = {0002-9939},
   MRCLASS = {30D50},
  MRNUMBER = {1641125},
MRREVIEWER = {Walter Bergweiler},
       DOI = {10.1090/S0002-9939-99-05175-8},
       URL = {https://doi.org/10.1090/S0002-9939-99-05175-8},
}

@misc{preprint,
      title={Average measure theoretic entropy for a family of expanding on average random Blaschke products}, 
      author={Cecilia González-Tokman and Renee Oldfield},
      year={2025},
      eprint={2505.09948},
      archivePrefix={arXiv},
      primaryClass={math.DS},
      howpublished = {\url{https://arxiv.org/abs/2505.09948}},
      note={To appear \textit{Ergodic Theory and Dynamical Systems}.}
}

@article {subadditive_MET,
    AUTHOR = {Gou\"{e}zel, S\'{e}bastien and Karlsson, Anders},
     TITLE = {Subadditive and multiplicative ergodic theorems},
   JOURNAL = {J. Eur. Math. Soc. (JEMS)},
  FJOURNAL = {Journal of the European Mathematical Society (JEMS)},
    VOLUME = {22},
      YEAR = {2020},
    NUMBER = {6},
     PAGES = {1893--1915},
      ISSN = {1435-9855},
   MRCLASS = {37H15 (32F45 47A35)},
  MRNUMBER = {4092901},
MRREVIEWER = {Ricardo F. Vila Freyer},
       DOI = {10.4171/jems/958},
zURL = {https://doi.org/10.4171/jems/958},
}

@incollection {parabolic,
    AUTHOR = {Contreras, Manuel D. and D\'{\i}az-Madrigal, Santiago and
              Pommerenke, Christian},
     TITLE = {Iteration in the unit disk: the parabolic zoo},
 BOOKTITLE = {Complex and harmonic analysis},
     PAGES = {63--91},
 PUBLISHER = {DEStech Publ., Inc., Lancaster, PA},
      YEAR = {2007},
      ISBN = {978-1-932078-73-2},
   MRCLASS = {30D05 (30-02)},
  MRNUMBER = {2387282},
}

@article {fletcher,
    AUTHOR = {Fletcher, Alastair},
     TITLE = {Unicritical {B}laschke products and domains of ellipticity},
   JOURNAL = {Qual. Theory Dyn. Syst.},
  FJOURNAL = {Qualitative Theory of Dynamical Systems},
    VOLUME = {14},
      YEAR = {2015},
    NUMBER = {1},
     PAGES = {25--38},
      ISSN = {1575-5460,1662-3592},
   MRCLASS = {30D05 (30J10 37F10)},
  MRNUMBER = {3326210},
MRREVIEWER = {Wayne\ Stewart\ Smith},
       DOI = {10.1007/s12346-015-0133-4},
       URL = {https://doi.org/10.1007/s12346-015-0133-4},
}

@article {martin,
    AUTHOR = {Martin, N. F. G.},
     TITLE = {On finite {B}laschke products whose restrictions to the unit
              circle are exact endomorphisms},
   JOURNAL = {Bull. London Math. Soc.},
  FJOURNAL = {The Bulletin of the London Mathematical Society},
    VOLUME = {15},
      YEAR = {1983},
    NUMBER = {4},
     PAGES = {343--348},
      ISSN = {0024-6093,1469-2120},
   MRCLASS = {30D50 (28D20)},
  MRNUMBER = {703758},
MRREVIEWER = {Kenneth\ Stephenson},
       DOI = {10.1112/blms/15.4.343},
       URL = {https://doi.org/10.1112/blms/15.4.343},
}

@article {baker_domains,
    AUTHOR = {Bergweiler, Walter},
     TITLE = {Singularities in {B}aker domains},
   JOURNAL = {Comput. Methods Funct. Theory},
  FJOURNAL = {Computational Methods and Function Theory},
    VOLUME = {1},
      YEAR = {2001},
    NUMBER = {1, [On table of contents: 2002]},
     PAGES = {41--49},
      ISSN = {1617-9447},
   MRCLASS = {30D05 (30F45 37F10 39B12)},
  MRNUMBER = {1931601},
       DOI = {10.1007/BF03320975},
       URL = {https://doi.org/10.1007/BF03320975},
}

@article {baker_domains2,
    AUTHOR = {Bara\'{n}ski, Krzysztof and Fagella, N\'{u}ria and Jarque,
              Xavier and Karpi\'{n}ska, Bogus\l awa},
     TITLE = {Escaping points in the boundaries of {B}aker domains},
   JOURNAL = {J. Anal. Math.},
  FJOURNAL = {Journal d'Analyse Math\'{e}matique},
    VOLUME = {137},
      YEAR = {2019},
    NUMBER = {2},
     PAGES = {679--706},
      ISSN = {0021-7670,1565-8538},
   MRCLASS = {37F10 (30D05)},
  MRNUMBER = {3938017},
MRREVIEWER = {Sanjay\ Kumar},
       DOI = {10.1007/s11854-019-0011-0},
       URL = {https://doi.org/10.1007/s11854-019-0011-0},
}

@article {benini_1,
    AUTHOR = {Benini, Anna Miriam and Evdoridou, Vasiliki and Fagella,
              N\'{u}ria and Rippon, Philip J. and Stallard, Gwyneth M.},
     TITLE = {Boundary dynamics for holomorphic sequences, non-autonomous
              dynamical systems and wandering domains},
   JOURNAL = {Adv. Math.},
  FJOURNAL = {Advances in Mathematics},
    VOLUME = {446},
      YEAR = {2024},
     PAGES = {Paper No. 109673, 51},
      ISSN = {0001-8708,1090-2082},
   MRCLASS = {30D05 (37F46)},
  MRNUMBER = {4736589},
MRREVIEWER = {Marco\ Abate},
       DOI = {10.1016/j.aim.2024.109673},
       URL = {https://doi.org/10.1016/j.aim.2024.109673},
}

@article {ergodic_inner_function,
    AUTHOR = {Aaronson, Jon},
     TITLE = {Ergodic theory for inner functions of the upper half plane},
   JOURNAL = {Ann. Inst. H. Poincar\'{e} Sect. B (N.S.)},
  FJOURNAL = {Annales de l'Institut Henri Poincar\'{e}. Section B. Calcul
              des Probabilit\'{e}s et Statistique. Nouvelle S\'{e}rie},
    VOLUME = {14},
      YEAR = {1978},
    NUMBER = {3},
     PAGES = {233--253},
      ISSN = {0020-2347},
   MRCLASS = {28D05 (30D50)},
  MRNUMBER = {508928},
MRREVIEWER = {Michael\ Lin},
}

@article {benini2,
    AUTHOR = {Benini, Anna Miriam and Evdoridou, Vasiliki and Fagella,
              N\'{u}ria and Rippon, Philip J. and Stallard, Gwyneth M.},
     TITLE = {Shrinking targets and recurrent behaviour for forward
              compositions of inner functions},
   JOURNAL = {J. Math. Pures Appl. (9)},
  FJOURNAL = {Journal de Math\'{e}matiques Pures et Appliqu\'{e}es.
              Neuvi\`eme S\'{e}rie},
    VOLUME = {208},
      YEAR = {2026},
     PAGES = {Paper No. 103864, 25},
      ISSN = {0021-7824,1776-3371},
   MRCLASS = {30D05 (30J05 37F10)},
  MRNUMBER = {5032916},
       DOI = {10.1016/j.matpur.2026.103864},
       URL = {https://doi.org/10.1016/j.matpur.2026.103864},
}

@incollection {lorentzen,
    AUTHOR = {Lorentzen, Lisa},
     TITLE = {Compositions of contractions},
      NOTE = {Extrapolation and rational approximation (Luminy, 1989)},
   JOURNAL = {J. Comput. Appl. Math.},
  FJOURNAL = {Journal of Computational and Applied Mathematics},
    VOLUME = {32},
      YEAR = {1990},
    NUMBER = {1-2},
     PAGES = {169--178},
      ISSN = {0377-0427,1879-1778},
   MRCLASS = {30D05 (30B70)},
  MRNUMBER = {1091787},
MRREVIEWER = {Heinrich\ Begehr},
       DOI = {10.1016/0377-0427(90)90428-3},
       URL = {https://doi.org/10.1016/0377-0427(90)90428-3},
}

@article {random_circle_endo,
    AUTHOR = {Goverse, V. P. H. and Homburg, A. J. and Lamb, J. S. W.},
     TITLE = {Intermittent {T}wo-{P}oint {D}ynamics at the {T}ransition to
              {C}haos for {R}andom {C}ircle {E}ndomorphisms},
   JOURNAL = {Comm. Math. Phys.},
  FJOURNAL = {Communications in Mathematical Physics},
    VOLUME = {407},
      YEAR = {2026},
    NUMBER = {5},
     PAGES = {Paper No. 87},
      ISSN = {0010-3616,1432-0916},
   MRCLASS = {Prelim},
  MRNUMBER = {5055737},
       DOI = {10.1007/s00220-026-05565-w},
       URL = {https://doi.org/10.1007/s00220-026-05565-w},
}

@article {synchronization_intermittency,
    AUTHOR = {Abbasi, Neda and Gharaei, Masoumeh and Homburg, Ale Jan},
     TITLE = {Iterated function systems of logistic maps: synchronization
              and intermittency},
   JOURNAL = {Nonlinearity},
  FJOURNAL = {Nonlinearity},
    VOLUME = {31},
      YEAR = {2018},
    NUMBER = {8},
     PAGES = {3880--3913},
      ISSN = {0951-7715,1361-6544},
   MRCLASS = {37E05 (28D05 37H10)},
  MRNUMBER = {3826118},
MRREVIEWER = {Byungik\ Kahng},
       DOI = {10.1088/1361-6544/aac637},
       URL = {https://doi.org/10.1088/1361-6544/aac637},
}

@article {simple_random_map,
    AUTHOR = {Yan, Jin and Majumdar, Moitrish and Ruffo, Stefano and Sato,
              Yuzuru and Beck, Christian and Klages, Rainer},
     TITLE = {Transition to anomalous dynamics in a simple random map},
   JOURNAL = {Chaos},
  FJOURNAL = {Chaos. An Interdisciplinary Journal of Nonlinear Science},
    VOLUME = {34},
      YEAR = {2024},
    NUMBER = {2},
     PAGES = {Paper No. 023128, 17},
      ISSN = {1054-1500,1089-7682},
   MRCLASS = {37E05 (37H10)},
  MRNUMBER = {4706371},
       DOI = {10.1063/5.0176310},
       URL = {https://doi.org/10.1063/5.0176310},
}

@article{Mane1991,
    AUTHOR = {Mañé, Ricardo and Doering, Claus I.},
    JOURNAL = {Ensaios Matemáticos},
    PAGES = {5-79},
    TITLE = {The Dynamics of Inner Functions},
    URL = {http://eudml.org/doc/186559},
    VOLUME = {3},
    YEAR = {1991},
}

@article {gelfert_synchronization,
    AUTHOR = {Gelfert, Katrin and Salcedo, Graccyela},
     TITLE = {Synchronization rates and limit laws for random dynamical
              systems},
   JOURNAL = {Math. Z.},
  FJOURNAL = {Mathematische Zeitschrift},
    VOLUME = {308},
      YEAR = {2024},
    NUMBER = {1},
     PAGES = {Paper No. 10, 35},
      ISSN = {0025-5874,1432-1823},
   MRCLASS = {37E10 (37B05 37H12 60F05)},
  MRNUMBER = {4780038},
MRREVIEWER = {Zhe\ Wang},
       DOI = {10.1007/s00209-024-03571-z},
       URL = {https://doi.org/10.1007/s00209-024-03571-z},
}

@article {hochman_furstenberg,
    AUTHOR = {Hochman, Michael and Solomyak, Boris},
     TITLE = {On the dimension of {F}urstenberg measure for {$SL_2(\Bbb R)$}
              random matrix products},
   JOURNAL = {Invent. Math.},
  FJOURNAL = {Inventiones Mathematicae},
    VOLUME = {210},
      YEAR = {2017},
    NUMBER = {3},
     PAGES = {815--875},
      ISSN = {0020-9910,1432-1297},
   MRCLASS = {37C85 (37C45 37D25)},
  MRNUMBER = {3735630},
MRREVIEWER = {Bal\'{a}zs\ B\'{a}r\'{a}ny},
       DOI = {10.1007/s00222-017-0740-6},
       URL = {https://doi.org/10.1007/s00222-017-0740-6},
}

@article {beardon,
    AUTHOR = {Beardon, Alan F.},
     TITLE = {Repeated compositions of analytic maps},
   JOURNAL = {Comput. Methods Funct. Theory},
  FJOURNAL = {Computational Methods and Function Theory},
    VOLUME = {1},
      YEAR = {2001},
    NUMBER = {1, [On table of contents: 2002]},
     PAGES = {235--248},
      ISSN = {1617-9447},
   MRCLASS = {30D05 (30C80 30D45 30F45)},
  MRNUMBER = {1931613},
MRREVIEWER = {Rich\ L.\ Stankewitz},
       DOI = {10.1007/BF03320987},
       URL = {https://doi.org/10.1007/BF03320987},
}

@article {keen_lakic,
    AUTHOR = {Keen, Linda and Lakic, Nikola},
     TITLE = {Random holomorphic iterations and degenerate subdomains of the
              unit disk},
   JOURNAL = {Proc. Amer. Math. Soc.},
  FJOURNAL = {Proceedings of the American Mathematical Society},
    VOLUME = {134},
      YEAR = {2006},
    NUMBER = {2},
     PAGES = {371--378},
      ISSN = {0002-9939,1088-6826},
   MRCLASS = {37F10 (30C35 30C70 30C75)},
  MRNUMBER = {2176004},
MRREVIEWER = {Volker\ Mayer},
       DOI = {10.1090/S0002-9939-05-08280-8},
       URL = {https://doi.org/10.1090/S0002-9939-05-08280-8},
}

@incollection {lorentzen2,
    AUTHOR = {Lorentzen, Lisa},
     TITLE = {Convergence of compositions of self-mappings},
      NOTE = {XII-th Conference on Analytic Functions (Lublin, 1998)},
   JOURNAL = {Ann. Univ. Mariae Curie-Sk\l odowska Sect. A},
  FJOURNAL = {Annales Universitatis Mariae Curie-Sk\l odowska. Sectio A.
              Mathematica},
    VOLUME = {53},
      YEAR = {1999},
     PAGES = {121--145},
      ISSN = {0365-1029},
   MRCLASS = {37F50 (30B70 30D05 40A15)},
  MRNUMBER = {1775541},
MRREVIEWER = {Stefano\ Marmi},
}

@article {eckman_ruelle,
    AUTHOR = {Eckmann, J.-P. and Ruelle, D.},
     TITLE = {Ergodic theory of chaos and strange attractors},
   JOURNAL = {Rev. Modern Phys.},
  FJOURNAL = {Reviews of Modern Physics},
    VOLUME = {57},
      YEAR = {1985},
    NUMBER = {3, part 1},
     PAGES = {617--656},
      ISSN = {0034-6861,1539-0756},
   MRCLASS = {58F11 (58-02 58F13)},
  MRNUMBER = {800052},
MRREVIEWER = {C.\ Tresser},
       DOI = {10.1103/RevModPhys.57.617},
       URL = {https://doi.org/10.1103/RevModPhys.57.617},
}

@article {blumenthal_young,
    AUTHOR = {Blumenthal, Alex and Young, Lai-Sang},
     TITLE = {Equivalence of physical and {SRB} measures in random dynamical
              systems},
   JOURNAL = {Nonlinearity},
  FJOURNAL = {Nonlinearity},
    VOLUME = {32},
      YEAR = {2019},
    NUMBER = {4},
     PAGES = {1494--1524},
      ISSN = {0951-7715,1361-6544},
   MRCLASS = {37H10 (37C40 37D45)},
  MRNUMBER = {3925383},
MRREVIEWER = {Romain\ Aimino},
       DOI = {10.1088/1361-6544/aafaa8},
       URL = {https://doi.org/10.1088/1361-6544/aafaa8},
}

@incollection {hyperbolic_metric,
    AUTHOR = {Beardon, A. F. and Minda, D.},
     TITLE = {The hyperbolic metric and geometric function theory},
 BOOKTITLE = {Quasiconformal mappings and their applications},
     PAGES = {9--56},
 PUBLISHER = {Narosa, New Delhi},
      YEAR = {2007},
      ISBN = {978-81-7319-807-6; 81-7319-807-1},
   MRCLASS = {30F45 (47H09)},
  MRNUMBER = {2492498},
}

@article {feigenbaum,
    AUTHOR = {Feigenbaum, Mitchell J.},
     TITLE = {Universal behavior in nonlinear systems},
   JOURNAL = {Los Alamos Sci.},
  FJOURNAL = {Los Alamos Science},
    NUMBER = {1},
    VOLUME = {1},
      YEAR = {1980},
     PAGES = {4--27},
   MRCLASS = {58F14 (70K50 76F99 82A05)},
  MRNUMBER = {608366},
MRREVIEWER = {Frederick\ R.\ Marotto},
}

@article {intermittency,
    AUTHOR = {Pomeau, Yves and Manneville, Paul},
     TITLE = {Intermittent transition to turbulence in dissipative dynamical
              systems},
   JOURNAL = {Comm. Math. Phys.},
  FJOURNAL = {Communications in Mathematical Physics},
    VOLUME = {74},
      YEAR = {1980},
    NUMBER = {2},
     PAGES = {189--197},
      ISSN = {0010-3616,1432-0916},
   MRCLASS = {58F13 (65L20)},
  MRNUMBER = {576270},
       URL = {http://projecteuclid.org/euclid.cmp/1103907981},
}

@article {furstenberg,
    AUTHOR = {Furstenberg, Harry},
     TITLE = {Noncommuting random products},
   JOURNAL = {Trans. Amer. Math. Soc.},
  FJOURNAL = {Transactions of the American Mathematical Society},
    VOLUME = {108},
      YEAR = {1963},
     PAGES = {377--428},
      ISSN = {0002-9947,1088-6850},
   MRCLASS = {60.08 (60.66)},
  MRNUMBER = {163345},
MRREVIEWER = {G.-C.\ Rota},
       DOI = {10.2307/1993589},
       URL = {https://doi.org/10.2307/1993589},
}

@article {malicet,
    AUTHOR = {Malicet, Dominique},
     TITLE = {Random walks on {${\rm Homeo}(S^1)$}},
   JOURNAL = {Comm. Math. Phys.},
  FJOURNAL = {Communications in Mathematical Physics},
    VOLUME = {356},
      YEAR = {2017},
    NUMBER = {3},
     PAGES = {1083--1116},
      ISSN = {0010-3616,1432-0916},
   MRCLASS = {60B15 (60B10 60G50)},
  MRNUMBER = {3719548},
MRREVIEWER = {Michael\ Voit},
       DOI = {10.1007/s00220-017-2996-5},
       URL = {https://doi.org/10.1007/s00220-017-2996-5},
}

@article {finititude_physical_measure,
    AUTHOR = {Barrientos, Pablo G. and Nakamura, Fumihiko and Nakano, Yushi
              and Toyokawa, Hisayoshi},
     TITLE = {Finitude of physical measures for random maps},
   JOURNAL = {Ast\'{e}risque},
  FJOURNAL = {Ast\'{e}risque},
    NUMBER = {459},
      YEAR = {2025},
     PAGES = {103},
      ISSN = {0303-1179,2492-5926},
      ISBN = {978-2-37905-214-9},
   MRCLASS = {37C40 (37A30 37A50 37C30 37H05 60J05)},
  MRNUMBER = {4938477},
MRREVIEWER = {Zhiming\ Li},
       DOI = {10.24033/ast.1248},
       URL = {https://doi.org/10.24033/ast.1248},
}

@article {kifer_random_diffeo,
    AUTHOR = {Brin, M. and Kifer, Yu.},
     TITLE = {Dynamics of {M}arkov chains and stable manifolds for random
              diffeomorphisms},
   JOURNAL = {Ergodic Theory Dynam. Systems},
  FJOURNAL = {Ergodic Theory and Dynamical Systems},
    VOLUME = {7},
      YEAR = {1987},
    NUMBER = {3},
     PAGES = {351--374},
      ISSN = {0143-3857,1469-4417},
   MRCLASS = {58G32 (58F18 60H25 60J27)},
  MRNUMBER = {912374},
MRREVIEWER = {Yu.\ E.\ Gliklikh},
       DOI = {10.1017/S0143385700004107},
       URL = {https://doi.org/10.1017/S0143385700004107},
}

@article {araujo_time_average,
    AUTHOR = {Ara\'{u}jo, V\'{\i}tor},
     TITLE = {Attractors and time averages for random maps},
   JOURNAL = {Ann. Inst. H. Poincar\'{e} C Anal. Non Lin\'{e}aire},
  FJOURNAL = {Annales de l'Institut Henri Poincar\'{e} C. Analyse Non
              Lin\'{e}aire},
    VOLUME = {17},
      YEAR = {2000},
    NUMBER = {3},
     PAGES = {307--369},
      ISSN = {0294-1449,1873-1430},
   MRCLASS = {37H20 (37D99 37G35)},
  MRNUMBER = {1771137},
MRREVIEWER = {Lorenzo\ J.\ D\'{\i}az},
       DOI = {10.1016/S0294-1449(00)00112-8},
       URL = {https://doi.org/10.1016/S0294-1449(00)00112-8},
}

@article {homburg_random_diffeo,
    AUTHOR = {Zmarrou, Hicham and Homburg, Ale Jan},
     TITLE = {Dynamics and bifurcations of random circle diffeomorphisms},
   JOURNAL = {Discrete Contin. Dyn. Syst. Ser. B},
  FJOURNAL = {Discrete and Continuous Dynamical Systems. Series B. A Journal
              Bridging Mathematics and Sciences},
    VOLUME = {10},
      YEAR = {2008},
    NUMBER = {2-3},
     PAGES = {719--731},
      ISSN = {1531-3492,1553-524X},
   MRCLASS = {37E10 (37E45 37G35 37H20)},
  MRNUMBER = {2425065},
MRREVIEWER = {Pawe\l \ G\'{o}ra},
       DOI = {10.3934/dcdsb.2008.10.719},
       URL = {https://doi.org/10.3934/dcdsb.2008.10.719},
}

@article {homburg_synchronization,
    AUTHOR = {Homburg, Ale Jan},
     TITLE = {Synchronization in minimal iterated function systems on
              compact manifolds},
   JOURNAL = {Bull. Braz. Math. Soc. (N.S.)},
  FJOURNAL = {Bulletin of the Brazilian Mathematical Society. New Series.
              Boletim da Sociedade Brasileira de Matem\'{a}tica},
    VOLUME = {49},
      YEAR = {2018},
    NUMBER = {3},
     PAGES = {615--635},
      ISSN = {1678-7544,1678-7714},
   MRCLASS = {37C05 (37D30)},
  MRNUMBER = {3862799},
MRREVIEWER = {Abbas\ Fakhari},
       DOI = {10.1007/s00574-018-0073-0},
       URL = {https://doi.org/10.1007/s00574-018-0073-0},
}

@article {homburg_kalle_intermittency,
    AUTHOR = {Homburg, Ale Jan and Kalle, Charlene},
     TITLE = {Iterated function systems of affine expanding and contracting
              maps on the unit interval},
   JOURNAL = {Adv. Math.},
  FJOURNAL = {Advances in Mathematics},
    VOLUME = {482},
      YEAR = {2025},
    NUMBER = {part B},
     PAGES = {Paper No. 110605, 48},
      ISSN = {0001-8708,1090-2082},
   MRCLASS = {37H20 (28A80 37A25 37E05 37H15)},
  MRNUMBER = {4975394},
MRREVIEWER = {Eugen\ Mihailescu},
       DOI = {10.1016/j.aim.2025.110605},
       URL = {https://doi.org/10.1016/j.aim.2025.110605},
}

@article {ruelle_takens,
    AUTHOR = {Ruelle, David and Takens, Floris},
     TITLE = {On the nature of turbulence},
   JOURNAL = {Comm. Math. Phys.},
  FJOURNAL = {Communications in Mathematical Physics},
    VOLUME = {20},
      YEAR = {1971},
     PAGES = {167--192},
      ISSN = {0010-3616,1432-0916},
   MRCLASS = {76.34 (34.00)},
  MRNUMBER = {284067},
MRREVIEWER = {J.\ Moser},
       URL = {http://projecteuclid.org/euclid.cmp/1103857186},
}

@article {CGTAQ,
    AUTHOR = {Gonz\'{a}lez-Tokman, Cecilia and Quas, Anthony},
     TITLE = {Stability and collapse of the {L}yapunov spectrum for
              {P}erron-{F}robenius operator cocycles},
   JOURNAL = {J. Eur. Math. Soc. (JEMS)},
  FJOURNAL = {Journal of the European Mathematical Society (JEMS)},
    VOLUME = {23},
      YEAR = {2021},
    NUMBER = {10},
     PAGES = {3419--3457},
      ISSN = {1435-9855,1435-9863},
   MRCLASS = {37D25 (30J10 37D35 37E10 37H15)},
  MRNUMBER = {4275477},
MRREVIEWER = {Eugen\ Mihailescu},
       DOI = {10.4171/jems/1096},
       URL = {https://doi.org/10.4171/jems/1096},
}

\end{document}